\documentclass[10pt,letter,reqno]{amsart}

\usepackage{setspace}
\usepackage{orcidlink}

\usepackage{tikz}
\usepackage{pgfplots}
\pgfplotsset{compat=1.15}
\usepackage{mathrsfs}
\usetikzlibrary{arrows}
\usepackage{pgfplots}

\usetikzlibrary{arrows.meta}
\usetikzlibrary{plotmarks}

\usepackage{amssymb,amsmath,amsfonts,amsthm,commath}
\usepackage{graphicx}
\usepackage{float}
\usepackage{framed}
\usepackage{calc,extarrows}
\usepackage{xcolor}

\usepackage{amsfonts}
\usepackage{amssymb}
\usepackage{graphicx}
\usepackage{newlfont}
\usepackage{mathrsfs}
\usepackage{graphics}
\usepackage{amsmath, amssymb, amsfonts,amsthm}
\usepackage{textcomp}
\usepackage{mathtools}
\usepackage{multicol}
\usepackage{bbm}
\usepackage{float}
\usepackage{enumerate}

\usepackage{hyperref}

\usepackage[numbers,sort&compress]{natbib}
\usepackage{marginnote}

\newtheorem{theorem}{Theorem}[section]
\newtheorem{cor}[theorem]{Corollary}

\newtheorem{prop}[theorem]{Proposition}
\newtheorem{definition}[theorem]{Definition}

\theoremstyle{remark}

\newcommand{\cd}{\ \Rightarrow \ }

\newcommand{\ed}{\ \stackrel{d}{=} \ }

\newcommand{\Rbold}{\mbox{${\mathbb R}$}}

\newcommand{\bE}{\mathbb{E}}
\newcommand{\bP}{\mathbb{P}}

\numberwithin{equation}{section}

\newtheorem{lemma}[theorem]{Lemma}

\newtheorem{remark}{Remark}[section]
\newcommand{\dist}{\mbox{dist}}
\newcommand{\Z}{\mathbb Z}
\newcommand{\Q}{\mathbb Q}

\newcommand{\R}{\mathbb R}

\renewcommand{\P}{\operatorname{\mathbb P}}

\DeclareMathOperator{\E}{\mathbb E}

\newcommand{\prob}{\xrightarrow{P}}

\newcommand{\disteq}{\overset{d}{=}}
\newcommand{\dd}{\mathrm{d}}

\newcommand{\Fi}{{\cal F}}

\newcommand{\Ni}{{\cal N}}

\newcommand{\Ui}{{\cal U}}

\newcommand{\Yi}{{\cal Y}}

\newcommand{\Exponential}[1]{\mathrm{Exponential}\left(#1\right)}

\begin{document}
	
	\title[Last Progeny Modified BRW]{Right-Most Position of a Last Progeny Modified Branching Random Walk}
	
	\author[Bandyopadhyay]{Antar Bandyopadhyay}
	\address[Antar Bandyopadhyay]{Theoretical Statistics and Mathematics Unit \\
		Indian Statistical Institute, Delhi \\ 
		7 S. J. S. Sansanwal Marg \\
		New Delhi 110016 \\
		INDIA}
	\address{Theoretical Statistics and Mathematics Unit, 
		Indian Statistical Institute, Kolkata
		203 B. T. Road, Kolkata 700108 \\ INDIA}
	\email{antar@isid.ac.in}     
	\author[Ghosh]{Partha Pratim Ghosh {\large\orcidlink{0000-0002-4801-4538}}}  
	\address[Partha Pratim Ghosh]{Institut f\"ur Mathematische Stochastik\\ Technische Universit\"at Braunschweig \\
	Universit\"atsplatz 2 \\ 
	38106 Braunschweig, GERMANY}         
	\email{p.pratim.10.93@gmail.com} 
	
	\date{February 14, 2025}
	
	\begin{abstract}
	 In this work, we consider a modification of the usual Branching Random Walk (BRW), where we give certain independent and identically distributed  (i.i.d.) displacements to all the particles at the $n$-th generation, which may be different from the driving increment distribution. We call this process \emph{last progeny modified branching random walk (LPM-BRW)}.
	Depending on the value of a parameter, 
	$\theta$, we classify the model into three distinct cases, namely, the boundary case, below the boundary case, and above the boundary case. Under very minimal assumptions on the underlying point process of the increments, we show that at the \emph{boundary case}, $\theta=\theta_0$, where $\theta_0$ is a parameter value associated with the displacement point process,  the maximum displacement converges to a limit after only an appropriate centering, which is of the form 	$c_1 n - c_2 \log n$. We give an explicit formula for the constants $c_1$ and $c_2$ and show that $c_1$ is exactly the same, while $c_2$ is $1/3$ of the corresponding constants of the usual BRW~\cite{Aide13}. We also characterize the limiting distribution. We further show that below the boundary, $\theta < \theta_0$, the logarithmic correction term is absent. For
above the boundary, $\theta > \theta_0$,  the logarithmic correction term is  exactly the same  as that of the classical BRW. 
For $\theta \leq \theta_0$, we further derive Brunet-Derrida -type results
of point process convergence of our LPM-BRW 
to a Poisson point process.
Our proofs are based on a novel method of coupling 
the maximum displacement with a \emph{linear statistic} associated 
with a more well-studied process 
in statistics, 
known as the \emph{smoothing transformation}.  
	\end{abstract}
	
	\keywords{Branching random walk, Bramson correction, derivative martingales, maximum operator, smoothing transformation.\\
	\textit{Published in: Journal of Theoretical Probability}, Volume 38(2): Paper No. 34 (2025).}

	
	\maketitle
	
\setlength{\parindent}{0em}
\setlength{\parskip}{1em}
	
	\section{Introduction}
	\label{Sec:Intro}
	
	\subsection{Introduction and background}
	\label{SubSec:Background}
	\emph{Branching random walk (BRW)} was introduced by Hammersley~\cite{Hamm74} in the early '70s. Over the last five decades, it has received a lot of attention from various researchers in probability theory and statistical physics. The model, as such, is very simple to describe.  
	It starts with one particle at the origin. 
	After a unit amount of time, the particle dies and gives birth to a number of similar particles, which are placed at possibly different locations on the real line $\Rbold$.  These particles at possibly different places on $\Rbold$ form the so-called first generation of the process and can be described through a point process, say $Z$ on $\Rbold$. After another unit time, each of the particles in the first generation behaves independently and identically as that of the parent, that is it dies, but before that, it produces a bunch of offspring particles which are displaced by independent copies of $Z$. The particles in generation one behave independently but identically of one another. The process then continues in the next unit of time and so on. The dynamics so produced is called a \emph{Branching random walk (BRW)}.

 Let $N:=Z(\R)$ be the offspring distribution of the underlying branching process. As will be
clear in the sequel (see Section~\ref{SubSec:Assumptions}), without loss of any generality 
throughout this article we will assume that $\P(N\geq 1)=1$. As otherwise, all of our results 
will hold when the process is supercritical and we conditioned on its survival. 
	
	Let $R_n$ denote the position of the \emph{right-most particle} in the generation
	$n$. In the seminal works,
	Hammersley~\cite{Hamm74}, 
	Kingman~\cite{King75}, and Biggins~\cite{Bigg76} 
	proved that under very minimal condition of the displacement point process $Z$, 
	\begin{equation}
		\frac{R_n}{n} \longrightarrow \gamma \mbox{\ \ a.s.},
		\label{Equ:BRW-SLLN}
	\end{equation}
	where $\gamma > 0$ is a constant associated with the displacement point process
	$Z$. 
	It is worth mentioning here that if we forget about the position of the particles and only keep count of the number of particles, then it forms a Galton-Watson branching process with progeny distribution given by 	$Z\left(\Rbold\right)$.  
	As noted in Aldous and Bandyopadhyay~\cite{AlBa05}, the arguments of
	Hammersley~\cite{Hamm74} can be used to claim that 
	if $\mbox{median}\left(R_{n+1}\right) - \mbox{median}\left(R_{n}\right)$ 
	remains bounded above, then the sequence of random variables
	$\left(R_n - \mbox{median}\left(R_{n}\right)\right)_{n \geq 0}$ remains
	tight. Similar arguments also appears in Dekking and Host \cite{DekkHo91}.
	
	From historical point-of-view, it is interesting to note here that Biggins~\cite{Bigg76} wrote: 
	\begin{quote}
		``Of course pride of place in the open problems goes to establishing more detailed results than~\eqref{Equ:BRW-SLLN} of the kinds that are already available for branching Brownian motion.''
	\end{quote}
	Indeed, McKean~\cite{McKe75} 
	showed that for similar continuous time version with 
	\emph{Branching Brownian Motion (BBM)},
	the maximum position, when centered by its median, converges weakly to a
	traveling wave solution. 
	Later Bramson~\cite{Bram78, Bram83} gave detailed order of the centering
	and showed that 
	an ``extra'' logarithmic term appears, which later was termed as the
	\emph{Bramson correction}. 
	Later Lalley and Sellke~\cite{LaSe87} gave a different probabilistic interpretation of the
	traveling wave limit through
	certain conditional limit theorem and using a new concept called the
	\emph{derivative martingales}.
	
	In a series of papers, 
	Bramson and Zeitouni~\cite{BrZe07, BrZe09} 
	showed that under fairly general
	conditions, $\left(R_n - \mbox{median}\left(R_{n}\right)\right)_{n \geq 0}$ remains
	tight. And in 2009, two groups of researchers,
	Hu and Shi~\cite{HuSh09} and Addario-Berry and Reed~\cite{AdRe09}
	independently proved that $\frac{R_n}{n}$ has a second-order fluctuation
	which was identified as $-\frac{3}{2} \log n$ in probability.  
	Finally, in 2013, A\"{i}d\'{e}kon~\cite{Aide13} proved that  $R_n-\gamma n+ \frac{3}{2} \log n$ converges in law to a randomly shifted Gumbel distribution, essentially settling the long-standing open problem of Biggins~\cite{Bigg76}. We refer to ~\cite{Shi2015} for an excellent 
	review of the classical and recent results on BRW. 
	
	In recent days more generally, it is expected that 
	this behavior for the maximum is shared by the universality class of 
	what is known as the ``\emph{log-correlated 
	fields}". We refer to \cite{Arg2017} for a detailed review of such generalization and results there in.

	In this work, we consider a modified version of the classical BRW. The modification is done at the last generation where we add \emph{i.i.d.} displacements	of a specific form. Since the modifications have been done only at the last generation, so we call this model 
	\emph{last progeny modified Branching Random Walk} or abbreviate it as
	\emph{LPM-BRW}. 
	 The model is described in more detail in the following subsection. We establish several results similar to
	 A\"{i}d\'{e}kon~\cite{Aide13} for our model and show that the  limit has the desired 
	 universality. Further work on large deviation for the same model and 
	 centered limits for a similar but inhomogeneous  displacements can be found in
	 ~\cite{Gh2022} and ~\cite{BaGh2023} respectively. 
	
	While we were preparing this manuscript Maillard and Mallein~\cite{MaiMal2021} considered a general 
	framework for characterizing the limiting distribution of what they called ``branching-type structure"
	via a fixed point of an operator 
    referred to as the \emph{branching convolution} introduced by Bertoin and Mallein 
   ~\cite{BerMal2019}	
	on the set of all 
	point processes endowed with an appropriate topology. They mentions in their paper 
	that our model is an example of 
	their general framework 
	(see fifth bullet point on page 2 of \cite{MaiMal2021}). 
	It is worth nothing here that \cite{MaiMal2021} does not provide any general proof of convergence
	after appropriate centering but gives characterization of the limit given convergence. 
	Our detailed analysis in this work provides a set of non-trivial and concrete cases where
	the result of \cite{MaiMal2021} may be applied for characterization of the limit.   
	
	\subsection{Model}
	\label{SubSec:Model}
	Let $Z=\sum_{j\geq1}\delta_{\xi_j}$ be a point process on $\R$ and $N:=Z(\R)<\infty$ a.s.  At the $0$-th generation, we start with an initial particle at the origin. At time $n\geq1$, each of the particles at generation $(n-1)$ gives birth to a random number of offspring distributed according to $N$. The offsprings are then given random displacements independently and according to a copy of the point process $Z$.

	For a particle $v$ we shall denote its generation by $|v|$, i.e., $|v| = n$ if $v$ belongs to the $n$-th generation. Let $S(v)$ denote the position of the particle $v$, which is the sum of all the displacements the particle $v$ and its ancestors have received. 
	The stochastic process 
	$\left\{S(v) \,\Big\vert\, \left\vert v \right\vert=n \,\right\}_{n \geq 0}$ is typically 
	referred to as the classical \emph{Branching Random Walk (BRW)}. The quantity of 
	interest is the maximum position, typically denoted by
	$R_n := \max_{\left\vert v \right\vert = n} S(v)$, is also the right-most position as
	discussed above.
	
	In our model, we introduce two parameters. One is a positive real number, which we denote by $\theta > 0$. The other one is a positively supported distribution, which we will denote by $\mu\in \mathcal{P}(\bar{\R}_+)$. The parameter $\theta$ should be thought of as a \emph{scaling parameter} for the extra displacement we give to each individual at the $n$-th generation. This extra displacement is as follows.  	At a generation $n \geq 1$, we give additional displacements to each of the particles at the generation $n$, which are of the form $\frac{1}{\theta} X_v := \frac{1}{\theta} \left(\log Y_v-\log E_v\right)$, where $\{Y_v\}_{|v|=n}$ are i.i.d. $\mu$, while $\{E_v\}_{|v|=n}$ are i.i.d. $\Exponential{1}$ and they are independent of each other and also of the process $\left(S(u)\right)_{\left\vert u \right\vert \leq n}$. We denote by $R_n^*(\theta, \mu)$ the maximum position of this 	\emph{last progeny modified branching random walk (LPM-BRW)}. If the parameters $\theta$ and $\mu$ are clear from the context, then we will simply write this as $R_n^*$.	A schematic of the process is given below. 
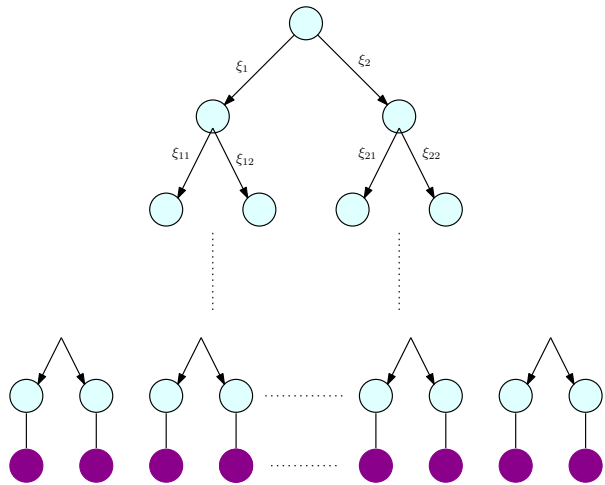
\begin{figure}[h!]
\centering
 \begin{tikzpicture}[line cap=round,line join=round,>=triangle 45,x=.3cm,y=.3cm]
\clip(-15,-24) rectangle (15,1);

\definecolor{cyan}{rgb}{0.0, 1.0, 1.0}
\definecolor{darkchestnut}{rgb}{0.6, 0.41, 0.38}

\node [circle, draw, fill=cyan, minimum size=0.45cm, inner sep=0pt, line width= 0.5pt] (V0) at (0,0){};

\node [circle, draw, fill=cyan, minimum size=0.45cm, inner sep=0pt, line width= 0.5pt] (V1) at (-4,-4){};
\node [circle, draw, fill=cyan, minimum size=0.45cm, inner sep=0pt, line width= 0.5pt] (V2) at (4,-4){};

\node [circle, draw, fill=cyan, minimum size=0.45cm, inner sep=0pt, line width= 0.5pt] (V11) at (-6,-8){};
\node [circle, draw, fill=cyan, minimum size=0.45cm, inner sep=0pt, line width= 0.5pt] (V12) at (-2,-8){};
\node [circle, draw, fill=cyan, minimum size=0.45cm, inner sep=0pt, line width= 0.5pt] (V21) at (2,-8){};
\node [circle, draw, fill=cyan, minimum size=0.45cm, inner sep=0pt, line width= 0.5pt] (V22) at (6,-8){};

    
\draw [-{Stealth[inset=0pt, length=0.15cm]},line width=   0.5pt] (V0) -- (V1) node[pos=0.4, left] {\scriptsize $\xi_{1}$}; 
\draw [-{Stealth[inset=0pt, length=0.15cm]},line width=   0.5pt] (V0) -- (V2) node[pos=0.4, right] {\scriptsize $\xi_{2}$};  

\draw [-{Stealth[inset=0pt, length=0.15cm]},line width=   0.5pt] (V1) -- (V11) node[pos=0.4, left] {\scriptsize $\xi_{11}$}; 
\draw [-{Stealth[inset=0pt, length=0.15cm]},line width=   0.5pt] (V1) -- (V12) node[pos=0.4, right] {\scriptsize $\xi_{12}$};  
\draw [-{Stealth[inset=0pt, length=0.15cm]},line width=   0.5pt] (V2) -- (V21) node[pos=0.4, left] {\scriptsize $\xi_{21}$}; 
\draw [-{Stealth[inset=0pt, length=0.15cm]},line width=   0.5pt] (V2) -- (V22) node[pos=0.4, right] {\scriptsize $\xi_{22}$};


\draw [dashed,line width=   0.5pt] (-2.25,-19) -- (2.25,-19);
\draw [dashed,line width=   0.5pt] (-2.25,-23) -- (2.25,-23);
\draw [dashed,line width=   0.5pt] (-4,-10) -- (-4,-13.5);
\draw [dashed,line width=   0.5pt] (4,-10) -- (4,-13.5);

\node [circle, draw, fill=cyan, minimum size=0.45cm, inner sep=0pt, line width= 0.5pt] (U1) at (-12.5,-15){};
\node [circle, draw, fill=cyan, minimum size=0.45cm, inner sep=0pt, line width= 0.5pt] (U2) at (-5.5,-15){};
\node [circle, draw, fill=cyan, minimum size=0.45cm, inner sep=0pt, line width= 0.5pt] (U3) at (5.5,-15){};
\node [circle, draw, fill=cyan, minimum size=0.45cm, inner sep=0pt, line width= 0.5pt] (U4) at (12.5,-15){};

\node [circle, draw, fill=cyan, minimum size=0.45cm, inner sep=0pt, line width= 0.5pt] (U11) at (-14,-19){};
\node [circle, draw, fill=cyan, minimum size=0.45cm, inner sep=0pt, line width= 0.5pt] (U12) at (-11,-19){};
\node [circle, draw, fill=cyan, minimum size=0.45cm, inner sep=0pt, line width= 0.5pt] (U21) at (-7,-19){};
\node [circle, draw, fill=cyan, minimum size=0.45cm, inner sep=0pt, line width= 0.5pt] (U22) at (-4,-19){};
\node [circle, draw, fill=cyan, minimum size=0.45cm, inner sep=0pt, line width= 0.5pt] (U31) at (4,-19){};
\node [circle, draw, fill=cyan, minimum size=0.45cm, inner sep=0pt, line width= 0.5pt] (U32) at (7,-19){};
\node [circle, draw, fill=cyan, minimum size=0.45cm, inner sep=0pt, line width= 0.5pt] (U41) at (11,-19){};
\node [circle, draw, fill=cyan, minimum size=0.45cm, inner sep=0pt, line width= 0.5pt] (U42) at (14,-19){};

\node [circle, draw, fill=darkchestnut, minimum size=0.45cm, inner sep=0pt, line width= 0.5pt] (W11) at (-14,-23){};
\node [circle, draw, fill=darkchestnut, minimum size=0.45cm, inner sep=0pt, line width= 0.5pt] (W12) at (-11,-23){};
\node [circle, draw, fill=darkchestnut, minimum size=0.45cm, inner sep=0pt, line width= 0.5pt] (W21) at (-7,-23){};
\node [circle, draw, fill=darkchestnut, minimum size=0.45cm, inner sep=0pt, line width= 0.5pt] (W22) at (-4,-23){};
\node [circle, draw, fill=darkchestnut, minimum size=0.45cm, inner sep=0pt, line width= 0.5pt] (W31) at (4,-23){};
\node [circle, draw, fill=darkchestnut, minimum size=0.45cm, inner sep=0pt, line width= 0.5pt] (W32) at (7,-23){};
\node [circle, draw, fill=darkchestnut, minimum size=0.45cm, inner sep=0pt, line width= 0.5pt] (W41) at (11,-23){};
\node [circle, draw, fill=darkchestnut, minimum size=0.45cm, inner sep=0pt, line width= 0.5pt] (W42) at (14,-23){};

    
\draw [-{Stealth[inset=0pt, length=0.15cm]},line width=   0.5pt] (U1) -- (U11);
\draw [-{Stealth[inset=0pt, length=0.15cm]},line width=   0.5pt] (U1) -- (U12);
\draw [-{Stealth[inset=0pt, length=0.15cm]},line width=   0.5pt] (U2) -- (U21);
\draw [-{Stealth[inset=0pt, length=0.15cm]},line width=   0.5pt] (U2) -- (U22);
\draw [-{Stealth[inset=0pt, length=0.15cm]},line width=   0.5pt] (U3) -- (U31);
\draw [-{Stealth[inset=0pt, length=0.15cm]},line width=   0.5pt] (U3) -- (U32);
\draw [-{Stealth[inset=0pt, length=0.15cm]},line width=   0.5pt] (U4) -- (U41);
\draw [-{Stealth[inset=0pt, length=0.15cm]},line width=   0.5pt] (U4) -- (U42);

\draw [-,line width=   0.5pt] (U11) -- (W11);
\draw [-,line width=   0.5pt] (U12) -- (W12);
\draw [-,line width=   0.5pt] (U21) -- (W21);
\draw [-,line width=   0.5pt] (U22) -- (W22);
\draw [-,line width=   0.5pt] (U31) -- (W31);
\draw [-,line width=   0.5pt] (U32) -- (W32);
\draw [-,line width=   0.5pt] (U41) -- (W41);
\draw [-,line width=   0.5pt] (U42) -- (W42);

\end{tikzpicture}
\caption{Last progeny modified branching random walk (LPM-BRW)}
\label{Fig:LPM-BRW}
\end{figure}

	\subsection{Assumptions}
	\label{SubSec:Assumptions}
	Before we state our assumptions, we introduce the following important quantities. 
	For a point process $Z = \mathop{\sum}\limits_{j=1}^{N} \delta_{\xi_j}$, we
	will write
	\[
	m\left(\theta\right) : = 
	\bE\left[\int_{{\mathbb R}}\! e^{\theta x} Z\left(\dd x\right)\right] =
	\bE\left[\sum_{j=1}^N e^{\theta \xi_j}\right],
	\]
	where $\theta \in \Rbold$, whenever the expectation exists. Naturally, $m$ is the 
	\emph{moment generating function} of the point process $Z$.
	Further, we define $\nu\left(t\right) := \log m\left(t\right)$ for 
	$t \in \Rbold$, whenever $m\left(t\right)$ is defined. 
	
	We now state our main assumptions. Throughout this paper, we will assume the following three conditions hold:
	\begin{itemize}
		\item[{\bf (A1)}]
		$m\left(\theta\right)<\infty$ 
		for all $\theta \in(-\vartheta,\infty)$ for some $\vartheta>0$.\\
		
		\item[{\bf (A2)}] 
		The point process $Z$ is \emph{non-trivial},  
		and the \emph{extinction probability} of 
		the underlying \emph{branching process} is $0$. In other words,  
		$\bP(N=1)<1$, $\bP( Z(\{t\})=N )<1$ for any $t\in\Rbold$,  
		and $\bP(N\geq1)=1 $.\\
		
		\item[{\bf (A3)}]
		$N$ has finite $(1+p)$-th moment for some $p>0$.
	\end{itemize}
	
	\begin{remark}
		{\bf (A1)} implies that $m$ is infinitely differentiable on $(-\vartheta,\infty)$. Together with {\bf (A3)}, it also implies that there exists $q>0$, such that, 
		for all $\theta\in[0,\infty)$,
		\begin{equation}
			\E\left[\left(\int_{\R} e^{\theta x}\,Z(\dd x) \right)^{1+q}\right]<\infty.\label{asmp_rem}
		\end{equation}
		Proof of this is given in the appendix (see Proposition~\ref{Prop:Equ-1.2}).
		
		Notice also that under Assumptions {\bf (A1)} and {\bf (A2)}, $\nu(t)$ is strictly convex in $\left(-\vartheta,\infty \right) $. Though this is a
		well-known fact, we are unable to find an exact reference for this. So a proof of this has been given in the Appendix as Proposition~\ref{Prop:Convexity_of_nu}.
	\end{remark}

	\subsection{Motivation}
	\label{SubSec:Motivation}
	Our main motivation to study this new LPM-BRW model 
	is what we will see in the sequel that, there is a nice \emph{coupling} of $R_n^*$ 
	with a \emph{linear statistic}, which is an additive martingale associated 
	with BRW (see Corollary~\ref{Thm:Coupling} for details). For such statistics, 
	asymptotics can be computed using various martingale techniques, some of
	which are known. This novel connection is indeed the reason
	the model intrigued us. As illustrated in this article, 
    our model is one example where this coupling technique works.	
	This connection is  novel and 
	we believe that 
    it has potential of many more applications.
	
	The other motivation and perhaps more straightforward one, is to
	be able to compare our results with the existing ones 
    in the context of the classical BRW
    (such as, asymptotics derived in~\cite{Aide13}).
    We see a difference appears in the constant factor in front of the
    Bramson correction (see Theorem~\ref{Thm:Boundary-Case}), but 
    the final weak limit remains the same. 
    This in turn shows that the centered asymptotic results are 
    heavily dependent on the displacements given at the end nodes, but 
    not the limit. While doing this comparison, 
    we also have been able to get the exact 
    constant for the centered limit 
    which was earlier not known
    (see Remark~\ref{Rem:Exact-Constant} for the details).

	\subsection{Outline}
	\label{SubSec:Outline}
	In Section~\ref{Sec:Results}, we state the main results. 
	Section~\ref{Sec:Coupling} provides our main tool: the
	coupling between the maximum statistic and a linear statistic.
	In Section~\ref{Sec:Auxilliary-Results}, we state and prove a few 
	asymptotic results about the associated \emph{linear statistic}, 
	which we later use in the proofs of the main results. 
	We end with Section~\ref{Sec:Proofs}, where we give all the details of the proofs.
	For the sake of 
	completeness, proofs of a few
	elementary results are provided in the Appendix.

	\section{Main Results}
	\label{Sec:Results}
	We start by defining a constant related to the underlying driving point process $Z$,
	which we denote by $\theta_0$. Let
	\[
	\theta_0:=\inf\left\{ \theta>0 : \frac{\nu(\theta)}{\theta}= \nu'(\theta) \right\}.
	\] 
	The fact that $\nu(\theta)$ is strictly convex ensures that the above set is at most singleton. If it is a singleton, then as illustrated in Figure~\ref{Fig:nu-theta_0}, $\theta_0$ is the  unique point in $(0,\infty)$ such that a  tangent line  from the origin to the graph of $\nu(\theta)$ touches the graph at $\theta=\theta_0$. And if it is empty, then by definition $\theta_0$ takes value $\infty$, and there does not exist any  tangent line  from the origin to the graph of $\nu(\theta)$ on the right half-plane. 
\begin{figure}[h!]
\begin{minipage}{.495\textwidth}
\centering
  \begin{tikzpicture}[line cap=round,line join=round,>=triangle 45,x=1.6cm,y=1.6cm]
\clip(-0.35,-0.35) rectangle (3.2,4);

\draw [{Stealth[inset=0pt, length=0.15cm]}-{Stealth[inset=0pt, length=0.15cm]},line width=   0.4pt] (-0.35,0) -- (3,0) node[right] {\scriptsize $x$};
    \draw [{Stealth[inset=0pt, length=0.15cm]}-{Stealth[inset=0pt, length=0.15cm]},line width=   0.4pt] (0,-0.35) -- (0,3.75) node[above] {\scriptsize $y$};
    
    \def\const{sqrt(2*ln(2))} 
    
    \draw[domain=0:2.4, smooth, line width=   0.4pt,  variable=\x] 
        plot ({\x}, {(0.5)*\x*\x + ln(2)}) ;
    \draw[domain=0:2.4, smooth, line width=   0.4pt, variable=\x] 
        plot ({\x}, {(\const)*\x}) ;
        
        \draw [dashed,line width=   0.4pt] (1.1774100225,0) -- (1.1774100225,1.3862943611);
        
   \begin{scriptsize}
   \node at (2,2.05) [ rotate=49.6580300383] {$y=\frac{\nu(\theta_0)}{\theta_0}x$};
     \node at (1.9,2.9) [ rotate=49.6580300383] {$y=\nu(x)$};
\draw[color=black] (-.1,-.1) node {$O$};
\draw[color=black] (1.1774100225,-.1) node {$x=\theta_0$};
\draw [fill=black] (1.1774100225,1.3862943611) circle (1pt);
\end{scriptsize}     
\end{tikzpicture}
\caption{Illustration of  $\theta_0$}
\label{Fig:nu-theta_0}
\end{minipage}
\begin{minipage}{.495\textwidth}
\centering
 \begin{tikzpicture}[line cap=round,line join=round,>=triangle 45,x=1.6cm,y=1.6cm]
\clip(-0.35,-0.35) rectangle (3.2,4);

\draw [{Stealth[inset=0pt, length=0.15cm]}-{Stealth[inset=0pt, length=0.15cm]},line width=   0.4pt] (-0.35,0) -- (3,0) node[right] {\scriptsize $x$} ;
    \draw [{Stealth[inset=0pt, length=0.15cm]}-{Stealth[inset=0pt, length=0.15cm]},line width=   0.4pt] (0,-0.35) -- (0,3.75) node[above] {\scriptsize $y$};
    
    \draw[domain=0.2:3, smooth, line width=   0.4pt,  variable=\x] 
        plot ({\x}, {(0.5)*(\x) + (ln(2))*(1/(\x))}) ;

        \draw [dashed,line width=   0.4pt] (1.1774100225,0) -- (1.1774100225,1.1774100225);
        
   \begin{scriptsize}
   \node at (2.5,1.75) [ rotate=20] {$y=\frac{\nu(x)}{x}$};
   \draw[color=black] (-.1,-.1) node {$O$};
\draw[color=black] (1.1774100225,-.1) node {$x=\theta_0$};
\draw [fill=black] (1.1774100225, 1.1774100225) circle (1pt);
\end{scriptsize}     
\end{tikzpicture}
\caption{Graph of $y = \nu(x)/x$}
\label{Fig:nu-theta_by_theta}
\end{minipage}
\end{figure}
\begin{remark}
\label{Rem:nutheta/theta}
It is worth noting that $\nu(\theta)/\theta$ is strictly  decreasing for $\theta\in(0,\theta_0)$ and strictly increasing for $\theta\in(\theta_0,\infty)$. Therefore, as shown in Figure~\ref{Fig:nu-theta_by_theta}, when $\theta_0$ is finite, it is the unique point of minimum for $\nu(\theta)/\theta$. 
\end{remark}

\begin{remark}
 \label{Rem:Free-Energy}
Note that 
\[
\frac{\nu(\theta)}{\theta}=\lim_{n\rightarrow\infty} \frac{1}{n\theta}\log\E\left[ W_n(\theta) \right],
\]
where $W_n(\theta)=W_n(\theta,0)$ is as defined in~\eqref{Equ:Def-W_n}. The quantity  $\nu(\theta)/\theta$ is often referred to as the \emph{``annealed free energy''}. The so-called \emph{``quenched free energy''}, denoted by $F(\theta)$, can be defined as
\[
 F(\theta):=\lim_{n\rightarrow\infty} \frac{1}{n\theta}\E\left[ \log W_n(\theta) \right].
 \]
Using Jensen's inequality, it is easy to see that they satisfy the  inequality
\[
F(\theta) \leq \frac{\nu(\theta)}{\theta}.
\]
\end{remark}

Whether $\theta_0$ is finite or infinite can be characterized by the fact that $\theta_0<\infty$, if and only if, 
	\[
	\lim_{\theta\rightarrow\infty} \left(\nu(\theta)-\theta\left(\lim_{x\rightarrow\infty}\nu'(x)\right)\right)<0.\\
	\]
	
In the sequel we will see that $\theta_0$ will be a point of \emph{phase transition} for our process. 
Indeed, it may be viewed as the \emph{critical inverse temperature} of the model, as it minimizes the 
limiting ``\emph{free energy}" (see Remark ~\ref{Rem:Free-Energy}).    	
 We thus classify our model into three different classes depending on the parameter
	$\theta$ is \emph{below}, \emph{equal}, or \emph{above} the quantity
	$\theta_0$. We term these as \emph{below the boundary case (BBC)}, 
	\emph{the boundary case (BC)}, and \emph{above the boundary case (ABC)}, respectively, rather than
	\emph{sub-critical}, \emph{critical} and \emph{super-critical}. We adopt to this terminology following
	Biggins and Kyprianou \cite{BiKy05} because our $\theta = \theta_0$ corresponds to what they call the 
	\emph{boundary case}.

	\subsection{Almost sure asymptotic limit}
	\label{SubSec:SLLN}
	Our first result is a \emph{strong law of large number}-type result, which is 
	similar to~\eqref{Equ:BRW-SLLN}.
	\begin{theorem}
		\label{Thm:SLLN}
		For every non-negatively supported probability $\mu\neq\delta_0$ that admits a 
		finite mean, almost surely
		\begin{equation}
			\frac{R_n^*(\theta,\mu)}{n}\rightarrow
			\begin{cases}
			\frac{\nu(\theta)}{\theta} &  \text{if }  \theta < \theta_0 \leq \infty; \\[0.25cm]
				\frac{\nu(\theta_0)}{\theta_0} &  \text{if }  \theta_0 \leq  \theta < \infty.
			\end{cases}
			\label{Equ:SLLN}
		\end{equation}
	\end{theorem}
	
	\begin{remark}
		Note that the almost sure limit remains the same as 
		$\frac{\nu(\theta_0)}{\theta_0}$ for both the BC and the ABC.
	\end{remark}
	
	\subsection{Centered asymptotic limits}
	\label{SubSec:Centered-Asymptotic-Limits}
	The centered asymptotic limits vary  in the three different cases depending on
	the value of the parameter $\theta$ as described above. We thus state the 
	results separately for the three cases.

	\subsubsection{The Boundary case ($\theta = \theta_0 < \infty$)}
	\label{SubSubSec:Boundary-Case-Results}
	\begin{theorem}
		\label{Thm:Boundary-Case}
		Assume that $\mu$ admits a finite mean, then
		there exists a random variable $H_{\theta_0}^{\infty}$, which 
		may depend on $\theta_0$,  such that,
		\begin{equation}
			R_n^* - \frac{\nu\left(\theta_0\right)}{\theta_0} n + \frac{1}{2 \theta_0} 
			\log n 
			\cd H_{\theta_0}^{\infty} + \frac{1}{\theta_0}\log \langle\mu\rangle,
			\label{Equ:Boundary-Limit}
		\end{equation}
		where $\langle\mu\rangle$ is the mean of $\mu$.

	\end{theorem}
	
	\begin{remark}
		Notice that the coefficient for the linear term, which is $\nu(\theta_0)/\theta_0$, is exactly the same as that of the centering of $R_n$, as proved by A\"{i}d\'{e}kon~\cite{Aide13}. However, the coefficient for the logarithmic term is $1/3$-rd of that of the centering of $R_n$, as shown by A\"{i}d\'{e}kon~\cite{Aide13}. The limiting distribution is also similar to that obtained by A\"{i}d\'{e}kon~\cite{Aide13}, which is a randomly shifted \emph{Gumbel distribution}.
    \end{remark}
	
	In fact, as we will see from the proof of the above theorem we also have the following result
	(see Section~\ref{Sec:Proofs}):
		
	\begin{theorem}
	\label{Thm:Boundary-Case-Version-II}
	Assume that $\mu$ admits a finite mean. Let 
		\begin{equation}
		\hat{H}_{\theta_0}^{\infty} 
		=
		\frac{1}{\theta_0}
		\left[ \log D_{\theta_0}^{\infty} 
		+ \frac{1}{2} \log \left( \frac{2}{\pi \sigma^2} \right) \right],
		\label{Equ:Def-H-theta0-infty}
		\end{equation}
		where
		\begin{equation}
		D_{\theta_0}^{\infty} \stackrel{\mbox{a.s.}}{=\joinrel=\joinrel=}
		\lim_{n \rightarrow \infty}
		-\frac{1}{m\left(\theta_0\right)^n}
		\sum_{\left\vert v \right\vert = n} \left(\theta_0 S(v) - 
		n \nu \left(\theta_0\right) \right)
		e^{\theta_0 S(v)},
		\label{Equ:Def-D-theta0-infty}
		\end{equation}
		
		\begin{equation}
		\sigma^2 := 
		\bE\left[  
		\frac{1}{m\left(\theta_0\right)}
		\sum_{\left\vert v \right\vert = 1} \left(\theta_0 S(v) -  \nu\left(\theta_0\right)\right)^2
		e^{\theta_0 S(v)}
		\right].
		\label{Equ:Def-sigma}
		\end{equation}
		Then
		\begin{equation}
			R_n^* - \frac{\nu\left(\theta_0\right)}{\theta_0} n + \frac{1}{2 \theta_0} 
			\log n - \hat{H}_{\theta_0}^{\infty} 
			\cd
		    \frac{1}{\theta_0} \left[\log \langle\mu\rangle -\log E\right],
			\label{Equ:Boundary-Limit-Version-II}
		\end{equation}
		where $E \sim \Exponential{1}$ and $\langle\mu\rangle$ is the mean of $\mu$.

	\end{theorem}
	
	\begin{remark}
	    We note here that the 
	    $H_{\theta_0}^{\infty}$ 
        in Theorem~~\ref{Thm:Boundary-Case} has the same distribution as
	    $\hat{H}_{\theta_0}^{\infty} - \frac{1}{\theta_0}\log E$, where 
	    $E \sim \Exponential{1}$ and is independent of
	    $\hat{H}_{\theta_0}^{\infty}$.
	    \end{remark}

\begin{remark}
	    \label{Rem:Exact-Constant}
		One advantage of the above result is that we have been able to identify the exact additive constant, which is $\frac{1}{2} \log \left( \frac{2}{\pi \sigma^2} \right)$, for the result in equation~\eqref{Equ:Boundary-Limit-Version-II}. As far as we know, this was not discovered in any of the earlier works.
	\end{remark}
	
	\begin{remark}
	It is worth mentioning here that, 
	$D_{\theta_0}^{\infty}$ is indeed the almost sure 
	limit of a \emph{derivative martingale} defined by 
\[
D_n:= -\sum_{|v|=n}(\theta_0S(v)-\nu(\theta_0)n)e^{\theta_0S(v)-\nu(\theta_0)n}
\]
The idea of the derivative martingale originates from Lalley and Sellke~\cite{LaSe87} and later it also appears in Biggins and Kyprianou~\cite{BiKy04} as well as in A\"{i}d\'{e}kon~\cite{Aide13}.$D_{\theta_0}^{\infty}>0$ a.s. under our assumptions and is a solution to a
\emph{linear recursive distributional equation (RDE)} given by 
\begin{equation}
\Delta\xlongequal{\text{ d }} \sum_{|v|=1}e^{\theta_0S(v)-\nu(\theta_0)} \Delta_v,
\label{Equ:RDE-for-D-theta0-infty}
\end{equation}
where $\Delta_v$'s are independent copies of $\Delta$.
	\end{remark}

	\subsubsection{Below the Boundary case ($\theta < \theta_0 \leq \infty$)}
	\label{SubSubSec:Below-Boundary-Case-Results}
	\begin{theorem}
		\label{Thm:Below-Boundary-Case}
		Assume that $\mu$ admits a finite mean, then
		for $ \theta < \theta_0 \leq \infty$, there exists a random variable  
		$H_{\theta}^{\infty}$, which 
		may depend on $\theta$,  such that,
		\begin{equation}
			R_n^* - \frac{\nu\left(\theta\right)}{\theta} n 
			\cd H^{\infty}_{\theta} + \frac{1}{\theta}\log \langle\mu\rangle,
			\label{Equ:Below-Boundary-Limit}
		\end{equation}
		where $\langle\mu\rangle$ is the mean of $\mu$.
	\end{theorem}
	
	\begin{remark}
		We note that in this case the \emph{logarithmic correction} disappears.
	\end{remark}
	
	Once again, just like in the boundary case, here too
	we have the following result also:
	
	\begin{theorem}
	\label{Thm:Below-Boundary-Case-Version-II}
Assume that $\mu$ admits a finite mean. Let 
    \[
\hat{H}_{\theta}^{\infty} = 
\frac{1}{\theta}\log D_{\theta}^{\infty},
\]
where
\begin{equation}
		D_{\theta}^{\infty} \stackrel{\mbox{a.s.}}{=\joinrel=\joinrel=}
		\lim_{n \rightarrow \infty}
		\frac{1}{m\left(\theta\right)^n}
		\sum_{\left\vert v \right\vert = n} 
		e^{\theta S(v)},
		\label{Equ:Def-D-theta-infty}
\end{equation}
which is also the mean $1$ solution of the 
following \emph{linear RDE}
\begin{equation}
\Delta\xlongequal{\text{ d }} \sum_{|v|=1}e^{\theta S(v)-\nu(\theta)} \Delta_v, 
\label{Equ:RDE-for-D-theta-infty}
\end{equation}
where $\Delta_v$'s are independent copies of $\Delta$. 
		Then
		\begin{equation}
			R_n^* - \frac{\nu\left(\theta\right)}{\theta} n - \hat{H}^{\infty}_{\theta} 
			\cd 
			\frac{1}{\theta}\left[\log \langle\mu\rangle -\log E\right],
			\label{Equ:Below-Boundary-Limit-Version-II}
		\end{equation}
     where $E \sim \Exponential{1}$ and $\langle\mu\rangle$ is the mean of $\mu$. 
	\end{theorem}
	
	\begin{remark}
		
	It is to be noted that the random variable $H_{\theta}^{\infty}$ in 
	Theorem~\ref{Thm:Below-Boundary-Case} has the same distribution as
	    $\hat{H}_{\theta}^{\infty} - \frac{1}{\theta}\log E$, where 
	    $E \sim \Exponential{1}$ and is independent of
	    $\hat{H}_{\theta}^{\infty}$.
	    \end{remark}
	    
	 \begin{remark}
		\label{Rem:RDE}   
	Biggins and Kyprianou~\cite{BiKy05} showed that under our assumptions, 
	the solutions to the linear RDE given in~\eqref{Equ:RDE-for-D-theta-infty}
	are unique up to a scale factor whenever they exist.  
	Therefore $D_{\theta}^{\infty}$ is 
	indeed the unique solution to the linear RDE 
	~\eqref{Equ:RDE-for-D-theta-infty} with mean $1$. 	
	\end{remark}

	\subsubsection{Above the Boundary case ($\theta_0 < \theta < \infty$)}
	\label{SubSubSec:Above-Boundary-Case-Results}
	\begin{theorem}
		\label{Thm:Above-Boundary-Case}
		Suppose $\mu = \delta_1$ and $Z$ is non-lattice, that is, $\P(Z(a\Z+b)=N)<1$ for all $a>0$ and $b\in\R$,  then
		for $\theta_0 < \theta < \infty$, there exists a constant $c_{\theta}\in\R$,
		which may depend on $\theta$,  such that,
		\begin{equation}
			R_n^* - \frac{\nu\left(\theta_0\right)}{\theta_0} n + \frac{3}{2 \theta_0} 
			\log n 
			\cd H_{\theta_0}^{\infty}+c_{\theta},
			\label{Equ:Above-Boundary-Limit}
		\end{equation}
		where $H_{\theta_0}^{\infty}$ is as in Theorem~\ref{Thm:Boundary-Case}.
	\end{theorem}

	\begin{remark}
		We would like to point out here that for the ABC,
		we have  been able to prove the centred limit only for $\mu = \delta_1$.  
		For technical reasons which will be clear from the proof, the general case
may give a different result. See Remark~\ref{Rem:remark_on_Mn} for more detail.
	\end{remark}

\subsection{Brunet-Derrida type results}
\label{SubSec:Brunet-Derrida-Limits}	
In this section, 
we present results of the type Brunet and Derrida~\cite{BrDe11} for convergence of the extremal point processes. Their conjecture 
for the classical  BRW was proven by 
Madaule~\cite{Mada17}. 
Here we present similar results for our LPM-BRW.
It is to be noted that the convergence of the point processes mentioned here is under the vague convergence topology on the set of all counting measures on $\R$.

Following Madaule~\cite{Mada17}, 
we now introduce point processes formed by the particles of appropriately re-centered branching random walks.
For any $\theta<\theta_0\leq \infty$, we consider
\begin{equation}
Z_n(\theta)=\sum_{|v|=n}\delta_{\left\{\theta S(v)-\log E_v-n\nu(\theta)-\log D_{\theta}^{\infty}\right\}}, 
\end{equation}
where $D_{\theta}^{\infty}$ is defined in  
Theorem~\ref{Thm:Below-Boundary-Case-Version-II}.  
And for $\theta=\theta_0<\infty$, we consider
\begin{equation}
Z_n(\theta_0)
=\sum_{|v|=n}\delta_{\left\{\theta_0 S(v)-\log E_v-n\nu(\theta_0)+\frac{1}{2}\log n-\log D_{\theta_0}^{\infty}-\frac{1}{2}\log\left(\frac{2}{\pi\sigma^2}\right)\right\}}, 
\end{equation}
where $D_{\theta_0}^{\infty}$ and $\sigma^2$ are as in Theorem~\ref{Thm:Boundary-Case-Version-II}.

Our first result is the weak convergence of  the point processes 
$\left(Z_n\left(\theta\right)\right)_{n \geq 0}$. 
\begin{theorem}
\label{Thm:Point-Process-Conv}
	For $\theta<\theta_0\leq\infty$ or $\theta=\theta_0<\infty$,
	\[Z_n(\theta)\xrightarrow{d} \Yi, \]
	where $\Yi$ is a Poisson point process on $\mathbb{R}$ with intensity measure $e^{-x}\,\dd x$. 
\end{theorem}
Following is a slightly weaker version of the above theorem, which is essentially
a point process convergence of the appropriately centered LPM-BRW model.
\begin{theorem}
\label{Thm:Point-Process-Weak-Conv}
For $\theta<\theta_0\leq\infty$,
	\[ 
	\sum_{|v|=n}\delta_{\left\{\theta S(v)-\log E_v-n\nu(\theta)\right\}}\xrightarrow{d} \sum_{j\geq1}\delta_{\zeta_j+\log D_{\theta}^{\infty}}, 
	\]
	and for $\theta=\theta_0<\infty$,
	\[ 
	\sum_{|v|=n}\delta_{\left\{\theta_0 S(v)-\log E_v-n\nu(\theta_0)+\frac{1}{2}\log n \right\}}\xrightarrow{d} \sum_{j\geq1} \delta_{ \zeta_j+\log D_{\theta_0}^{\infty}+\frac{1}{2}\log\left(\frac{2}{\pi\sigma^2}\right)}, 
	\]
	where $\Yi=\sum_{j\geq1}\delta_{\zeta_j}$ is a  Poisson point process on $\R$ with intensity measure $e^{-x}\,\dd x$, which is independent of the BRW.
\end{theorem}

Now, we denote $\Yi_{\max}$ as the right-most position of the point process $\Yi$, and we write $\overline{\Yi}$ as the point process $\Yi$ seen from its right-most position, 
that is,
\[ \overline{\Yi} = \sum_{j\geq1}\delta_{ \zeta_j-\Yi_{\max}}. \]
The following result is an immediate corollary of the above theorem, which
confirms that the \emph{Brunet-Derrida Conjecture} holds for our model 
when $\theta<\theta_0\leq\infty$ or $\theta=\theta_0<\infty$.
\begin{theorem}
\label{Thm:BD-Conj}
	For $\theta<\theta_0\leq\infty$ or $\theta=\theta_0<\infty$,
	\[\sum_{|v|=n}\delta_{\left\{\theta S(v)-\log E_v-\theta R_n^*(\theta,\delta_1) \right\}}\xrightarrow{d} \overline{\Yi}. \]
\end{theorem}

\begin{remark}
	 Madaule~\cite{Mada17} showed the convergence of the centered point process, obtained in the classical setup, to a decorated Poisson point process. As defined in~\cite{Mada17}, a decorated Poisson point process can be described as follows: Let $\mathfrak{Z} = \sum_{i \geq 1} \delta_{\zeta_i}$ be a Poisson point process with intensity $\lambda e^{-\alpha x} \,\dd x$, and let $\{\mathfrak{X}_i\}_{i \geq 1}$ be independent copies of a point process $\mathfrak{X}$, where $\mathfrak{X}_i = \sum_{j \geq 1} \delta_{\chi_{i,j}}$. Then, the point process $\mathcal{Q} = \sum_{i \geq 1} \sum_{j \geq 1} \delta_{\zeta_i + \chi_{i,j}}$ is called a decorated Poisson point process with decoration $\mathfrak{X}$. In Madaule's work~\cite{Mada17} the distribution of the decoration was left undescribed, which was later described in Mallein~\cite{Mall18}. It is worth noting that, in our case, the decoration disappears. This is due to the fact that for both BC and BBC,
	\begin{equation}
	\max_{|v|=n} \frac{e^{\theta S(v)}}{W_n(\theta)}\prob 0,
	\label{Equ:dec_disappear}
	\end{equation}
	as mentioned in~\eqref{mnlim1} and~\eqref{mnlim2}.
	However, as noted in Remark~\ref{Rem:remark_on_Mn},~\eqref{Equ:dec_disappear} does not hold for the ABC. This added complication is the main reason that the results for the ABC remain open.
\end{remark}

\begin{remark}
The point process $\overline{\Yi}$ can be described explicitly in the following way: Let $\Ni=\sum_{j\geq 1} \delta_{\mathfrak z_j}$ be a homogeneous Poisson point process   on $\R_+$ with intensity $1$ and $E\sim\Exponential{1}$ be independent of $\Ni$. Then
\[
\overline{\Yi} \disteq \delta_0+ \sum_{j\geq 1} \delta_{-\log\left(1+ (\mathfrak z_j/E)\right)}.
\]
\end{remark}

	\section{Coupling between a Maximum and a Linear Statistic}
	\label{Sec:Coupling}
	We start by defining a few operators on the space of probabilities which will help us to state and prove the coupling. 
	In the sequel, $\mathcal{P}(A)$ will mean the set of all probabilities on a measurable space $(A, \mathcal{A})$, $\bar{\R}=[-\infty,\infty]$, $\bar{\R}_+=[0,\infty]$ and 
	$\mbox{dist}(X)$ represents the distribution of a random variable $X$.  Let us also recall that 
$Z=\sum_{j\geq1}\delta_{\xi_j}$ denotes a point process on $\R$ and $N:=Z(\R)<\infty$ a.s.

	\begin{definition}[{\bf Maximum Operator}]
		\label{Def:Max-Operator}
		The operator $M_Z:\mathcal{P}(\bar{\R})\rightarrow\mathcal{P}(\bar{\R})$ defined by 
		\[
		M_Z(\eta) =\mbox{dist}\left(\max_j\{\xi_j+X_j\}\right),
		\]
		where $\{X_j\}_{j\geq1}$ are i.i.d. $\eta\in\mathcal{P}(\bar{\R})$ and are independent of $Z$, will be called the \emph{Maximum Operator}.
	\end{definition}	
	
	\begin{remark}
		Observe that 
		$M_Z^n(\eta)$ is the distribution of the maximum of the positions of the
		particles after adding i.i.d. displacements from $\eta$ to the particles at $n$-th generation:
		\[
		M_Z^n(\eta) =\mbox{dist}\left(\max_{|v|=n}\left\{S(v)+X_v\right\}\right).
		\]
		In particular, $R_n \sim M_Z^n\left(\delta_0\right)$ and
		$R_n^* \sim M_Z^n\left(\eta\right)$, where
		$\eta$ is the distribution of $\frac{1}{\theta}\log (Y_v/E_v)$ for a particle $v$ at generation $n$. 
	\end{remark}

	\begin{definition}	[{\bf Linear Operator}]
		\label{Def:Linear-Operator}
		The operator $L_Z:\mathcal{P}(\bar{\R}_+)\rightarrow\mathcal{P}(\bar{\R}_+)$ defined by
		\[  
		L_Z(\mu)=\mbox{dist}\left(\sum_{j\geq1}e^{\xi_j} Y_j\right),   
		\]
		where $\{Y_j\}_{j\geq1}$ are i.i.d.  $\mu\in\mathcal{P}(\bar{\R}_+)$ and are independent of $Z$, will be called the \emph{Linear} or \emph{Smoothing Operator}.
	\end{definition}
	
	\begin{remark}
		Observe that $L_Z^n(\mu)$ is the distribution of $\sum_{|v|=n}e^{S(v)} Y_v$.
	\end{remark}

	\begin{definition}	[{\bf Link Operator}]
		\label{Def:Link-Operator}
		The operator 
		$\mathcal{E}:\mathcal{P}(\bar{\R}_+)\rightarrow\mathcal{P}(\bar{\R})$  
		defined by
		\[
		\mathcal{E}(\mu) =\mbox{dist}\left(\log\frac{Y}{E}\right),
		\]
		where $E\sim \Exponential{1}$ and  $Y\sim\mu\in\mathcal{P}(\bar{\R}_+)$ and they are independent, will be called the \emph{Link Operator}.
	\end{definition}

	\begin{definition}
		For $a\geq0$ and $b\in \R$, the operator $\Xi_{a,b}$ on the set of all point processes is defined by
		\[
		\Xi_{a,b}(\mathcal{Z})=\sum_{j\geq1} \delta_{a\zeta_j-b},
		\]
		where $\mathcal{Z}=\sum_{j\geq1} \delta_{\zeta_j}$. Sometimes we may denote $\Xi_{a,0}$ by $\Xi_a$ for notational simplicity.
	\end{definition}
	
	The following result is one of the most important observations, and it links the operators defined above. As an immediate corollary, we get a very useful coupling between the LPM-BRW and the linear statistic associated with the \emph{linear operator}. 
	\begin{theorem}[Transforming Relationship]
		\label{Thm:Transforming-Relationship}
		For all $n \geq 1$,
		\begin{equation}
		M_Z^n\circ\mathcal{E}=\mathcal{E}\circ L_Z^n. 
		\label{Equ:Transforming-Relationship}
		\end{equation}
	\end{theorem}
	
	\begin{proof}
	    We first note that it is enough to show that  equation 
	  ~\eqref{Equ:Transforming-Relationship} holds for $n=1$, as 
	    the general case then follows by a trivial induction. To this end, 
		let $Z=\sum_{j\geq1}\delta_{\xi_{j}}$, $\{E_j\}_{j\geq1}$ are i.i.d. $\Exponential{1}$,  $\{Y_j\}_{j\geq1}$ are i.i.d. $\mu$, and they are independent of each other. Now,
		\begin{align}
			M_Z\circ\mathcal{E}(\mu)&=\dist\left(\max_j\left( \xi_j+\log\frac{Y_j}{E_j}\right)\right) \nonumber\\
			&=\dist\left(\max_j\left(\log\frac{e^{\xi_j}Y_j}{E_j}\right)\right)\nonumber\\
			&=\dist\left(-\log\left(\min_j\frac{E_j}{e^{\xi_j}Y_j}\right)\right)=\dist \left(-\log \frac{E_1}{\sum_{j\geq1}e^{\xi_j}Y_j}\right)=\mathcal{E}\circ L_Z(\mu). 
			\label{Equ:Tranforming_Relation_Proof}
		\end{align}
		 Note that the second-to-last equality in~\eqref{Equ:Tranforming_Relation_Proof} comes from the fact that, conditionally on $Z$ and $\{Y_j\}_{j\geq 1}$, the random variables $\left\{\frac{E_j}{e^{\xi_j}Y_j}\right\}_{j\geq 1}$ are independent and the conditional distribution of 
		$\frac{E_j}{e^{\xi_j}Y_j}$ is $\Exponential{ e^{\xi_j}Y_j }$. 
		Thus, using standard properties of exponential distribution, we conclude that 
		conditionally on $Z$ and $\{Y_j\}_{j\geq 1}$, 
		the distribution of $\min_j  \frac{E_j}{e^{\xi_j}Y_j}$ is 
		$\Exponential{ \sum_{j\geq1}e^{\xi_j}Y_j }$.
	\end{proof}

	 \begin{cor}
		\label{Thm:Coupling}
		Let $\theta > 0$ and $\mu \in \mathcal{P}(\bar{\R}_+)$. Then for any $n \geq 1$,
		\begin{equation}
		\theta R_n^*\left(\theta, \mu\right) \ed \log Y_n^{\mu}(\theta) - \log E,
		\label{Equ:Coupling_dist_equality}
		  \end{equation}
		where $Y_n^{\mu}(\theta) := \sum_{|v|=n} e^{\theta S(v)} Y_v$, $\{Y_v\}_{|v|=n}$ are i.i.d. $\mu$, and 
		$E \sim \Exponential{1}$
		and is independent of $Y_n^{\mu}$. In other words,
		\begin{equation}
        \mathbb{P}\left( R_n^*(\theta,\mu) \leq x\right) = \mathbb{E}\left[ e^{-e^{\theta x} \sum_{|v|=n} e^{\theta S(v)} Y_v } \right].
        \label{Equ:Coupling_exact_dist}
        \end{equation}
	\end{cor}
	
	\begin{proof}
	    Observe that, 
		\begin{align*}
			\mbox{dist}\left(\theta R_n^*\left(\theta, \mu\right)\right)
			& =  M_{\Xi_{\theta}(Z)}^n\circ\mathcal{E}(\mu)\\
			&    =  \mathcal{E}\circ L_{\Xi_{\theta}(Z)}^n(\mu)
			  =\mbox{dist} \left(\log Y_n^{\mu} - \log E\right).
		\end{align*}
	\end{proof}

\section{A Few Auxiliary Results on the Linear Statistic}
\label{Sec:Auxilliary-Results}

In this section, we provide a few convergence results related to the 
\emph{linear operator}, $L_Z^n$, as defined in the previous section and associated 
\emph{linear statistic}, which is defined in the sequel 
(see equation~\eqref{Equ:Def-W_n}).

We start by observing that if we 
consider the point process $\Xi_{\theta,\nu_Z(\theta)}(Z)$, then
	\[\nu_{\Xi_{\theta,\nu_Z(\theta)}(Z)}(\alpha)=\log \E\left[\int_{\R} e^{\alpha\theta x-\alpha\nu_Z(\theta)}\,Z(\dd x) \right]
	=\nu_Z(\alpha\theta)-\alpha\nu_Z(\theta).\]
	Differentiating this with respect to $\alpha$, we get
	\[ \nu_{\Xi_{\theta,\nu_Z(\theta)}(Z)}'(\alpha)=\theta\nu_Z'(\alpha\theta)-\nu_Z(\theta). \]
	Now, taking $\alpha=1$, we have
	$\nu_{\Xi_{\theta,\nu_Z(\theta)}(Z)}(1)=0$, and
	\[
	\nu_{\Xi_{\theta,\nu_Z(\theta)}(Z)}'(1)=\theta\nu_Z'(\theta)-\nu_Z(\theta)\,
	\begin{cases}
	>0&\text{if }\theta_0<\theta<\infty;\\
		=0&\text{if }\theta=\theta_0<\infty;\\
		<0&\text{if }\theta<\theta_0\leq \infty.
	\end{cases}
	\]
	Therefore, using \cite[Theorem~1.6]{Liu98}, we have
	\begin{equation}
	L_{\Xi_{\theta,\nu_Z(\theta)}(Z)}^n(\mu)\xrightarrow{w} 
	\begin{cases}
	\delta_0& \text{if }\theta=\theta_0<\infty;\\
		\mu_{\theta}^{\infty}& \text{if }\theta<\theta_0\leq \infty,
	\end{cases}
	 \label{Equ:LZnConv}
	\end{equation}
	where for all $\theta<\theta_0$, $\mu_{\theta}^{\infty}\neq\delta_0$ is a fixed point of $ L_{\Xi_{\theta,\nu_Z(\theta)}(Z)} $  and  has the same mean as $\mu$. Since $\mu_{\theta}^{\infty}\neq \delta_0$ is a fixed point of $ L_{\Xi_{\theta,\nu_Z(\theta)}(Z)} $, we also have  $\mu_{\theta}^{\infty}(\{0\})=0$ for all $\theta<\theta_0$. 
	
	We now define the \emph{linear statistic} associated with the
	linear operator $L_Z^n$. 
	\begin{equation}
	W_n(a,b):=\sum_{|v|=n}e^{aS(v)-nb}.
	\label{Equ:Def-W_n}
	\end{equation}
	To simplify the notations, sometimes we may write $W_n(a,0)$ as  $W_n(a)$.
	From the definition of the operator $L$, we get that
	\[
	L_{\Xi_{a,b}(Z)}^n(\delta_1)=\dist\left( W_n(a,b) \right).
	\]
	Since $\{W_n(\theta,\nu_Z(\theta))\}_{n\geq1}$ is a non-negative martingale, it converges a.s. Therefore~\eqref{Equ:LZnConv} implies that almost surely,
	\begin{equation}
		\label{wnliu}
		W_n(\theta,\nu_Z(\theta))\rightarrow
		\begin{cases}
			0& \text{if }\theta=\theta_0<\infty;\\
			D_{\theta}^{\infty}& \text{if }\theta<\theta_0\leq\infty,
                 \end{cases}
	\end{equation}
	for some positive random variable $D_{\theta}^{\infty}$ with $\E[D_{\theta}^{\infty}]=1$,
	and the distribution of $D_{\theta}^{\infty}$ is a solution to the linear RDE~\eqref{Equ:RDE-for-D-theta-infty}.
	
    The following proposition provides convergence results of $W_n(a,b)$ for various values of $a$ and $b$.
	\begin{prop}
		For any $a>0$ and $b\in\R$, almost surely
		\label{wnprop}
		\[
		W_n(a,b)
		\rightarrow
		\left\{
		\begin{array}{@{}lll}
			0 & \text{if } a<\theta_0,\, b>\nu(a);\,\,&(i)\\[.1cm]
			D_a^{\infty} & \text{if } a<\theta_0,\, b=\nu(a);\,\,&(ii)\\[.1cm]
			\infty & \text{if } a<\theta_0,\, b<\nu(a);\,\,&(iii)\\[.1cm]
			0 & \text{if } \theta_0<\infty,\,a\geq\theta_0,\, b\geq a\nu(\theta_0)/\theta_0;\,\,&(iv)\\[.1cm]
			\infty & \text{if }  \theta_0<\infty,\,a\geq\theta_0,\, b<a\nu(\theta_0)/\theta_0.\,\,&(v)
		\end{array}
		\right.
		\]
	\end{prop}
To prove this proposition, we use the following elementary result. We provide the proof for sake of completeness. 
\begin{lemma}
	\label{convx}
	Let $f:[0,\infty)\rightarrow \R$ be a continuously differentiable convex function  and $\mathbb{S}$ be a convex subset of $[0,\infty)\times \R$ satisfying
	\begin{itemize}
		\item $(x,y)\in \mathbb{S}$ for all $0<x< x_0$ and $y> f(x)$ and
		\item  $(x,y)\notin \mathbb{S}$ for all $0<x< x_0$ and $y< f(x)$,
	\end{itemize}
	for some $x_0>0$. Then
	\[
	\mathbb{S}\subseteq\left\{ (x,y): y\geq T_{x_0}(x) \right\},
	\]
	where $T_{x_0}(\cdot)$ denotes the  tangent line  to $f$ at $x_0$.
\end{lemma}
	\begin{proof}
		We define a function $g:[0,\infty)\rightarrow\bar{\R}$ as
		\[
		g(x)=\inf\left\{ y:(x,y)\in\mathbb{S} \right\}.
		\]
		We first show that $g$ is convex. Take any $x_1$, $x_2$ such that $g(x_1)$, $g(x_2)<\infty$. By definition of $g$, for every $\epsilon>0$, there exist $y_1<g(x_1)+\epsilon$ and $y_2<g(x_2)+\epsilon$ such that $(x_1,y_1),(x_2,y_2)\in\mathbb{S}$. So for any $\alpha\in(0,1)$, $(\alpha x_1+(1-\alpha)x_2,\alpha y_1+(1-\alpha)y_2)\in\mathbb{S}$. Therefore
		\[
		g(\alpha x_1+(1-\alpha)x_2)\leq \alpha y_1+(1-\alpha)y_2< \alpha g(x_1)+(1-\alpha)g(x_2)+\epsilon.
		\]
		As $\epsilon>0$ is arbitrary, we have 
		\[
		g(\alpha x_1+(1-\alpha)x_2)\leq  \alpha g(x_1)+(1-\alpha)g(x_2),
		\]
		and this is true for all  $\alpha\in(0,1)$. Therefore $g$ is convex.
		
		Let $T_x(.)$ be the  tangent line  to $ f $ at $x$. Since $f$ is continuously differentiable, $T_x$ converges pointwise to $T_{x_0}$ as $x\rightarrow x_0$. Note that $g=f$ in $(0,x_0)$. Therefore, for all $x\in(0,x_0)$,  $T_x$ is also the  tangent line  to $g$ at $x$. Since $g$ is convex, we have $g\geq T_x$ for all $x\in(0,x_0)$. Hence, $g\geq T_{x_0}$. This completes the proof.
	\end{proof}

	\begin{proof}[Proof of Proposition~\ref{wnprop}]
		\textit{Proof of (i),(ii) and (iii)}.
			Noting that
			\[W_n(a,b)=W_n(a,\nu(a))\cdot e^{n(\nu(a)-b)}\]
			(i), (ii) and (iii) follows from~\eqref{wnliu}.
		
		\textit{Proof of (iv)}.
			For $a\geq\theta_0$, we have
			\begin{align*}
				W_n(a,b)= \sum_{|v|=n}e^{aS(v)-nb}
				&\leq	\left(\sum_{|v|=n}e^{\left(aS(v)-nb\right)\theta_0/a}\right)^{a/\theta_0}\\[.25cm]
				&=W_n\left(\theta_0,b\theta_0/a\right)^{a/\theta_0}\\[.25cm]
				&= \left(W_n\left(\theta_0,\nu(\theta_0)\right)\cdot e ^{n\left(\nu(\theta_0)-b\theta_0/a\right)}\right)^{a/\theta_0}
			\end{align*}
			Since $W_n(a,b)$ is non-negative, using~\eqref{wnliu}, we get that for $a\geq\theta_0$ and $b\theta_0/a\geq \nu(\theta_0)$,
			\[W_n(a,b)\rightarrow 0 \text{ a.s.} \]
		
		\textit{Proof of (v)}.
			Using (i) and (iii), we know that there exists $\mathcal{N}\subset\Omega$ with $\P(\mathcal{N})=0$ such that for all $\omega\notin\mathcal{N}$ and $(a,b)\in [(0,\theta_0)\times\R] \cap \Q^2$,
			\[
			W_n(a,b)(\omega)
			\rightarrow
			\begin{cases}
				0 & \text{if }  b>\nu(a);\\
				\infty & \text{if }  b<\nu(a).                         
			\end{cases}
			\]
			For	any $\omega\notin\mathcal{N}$ and any subsequence $\{n_k\}$, we define
			\[
			\mathbb{S}\left(\{n_k\},\omega\right)=\left\{(c,d):\limsup_{k\rightarrow\infty}W_{n_k}(c,d)(\omega)<\infty \right\}.
			\]
			Now, suppose $(c_1,d_1),(c_2,d_2)\in\mathbb{S}\left(\{n_k\},\omega\right)$. Then for any $\alpha\in(0,1)$,
			\begin{align*}
				&\limsup_{k\rightarrow\infty} W_{n_k}\left(\alpha c_1+(1-\alpha)c_2,\alpha d_1+(1-\alpha)d_2\right)(\omega)\\[0.25cm]
				=\,\, & \limsup_{k\rightarrow\infty} \sum_{|v|=n_k}\exp\left(\alpha\left[c_1S(v)(\omega)-n_kd_1\right]+(1-\alpha)\left[c_2S(v)(\omega)-n_kd_2\right]\right)\\[0.25cm]
				\leq\,\, & \alpha\left[\limsup_{k\rightarrow\infty}\sum_{|v|=n_k}\exp\left(c_1S(v)(\omega)-n_kd_1\right)\right]
				\\[0.25cm]
				&\qquad\qquad +
				(1-\alpha)\left[\limsup_{k\rightarrow\infty}\sum_{|v|=n_k}\exp\left(c_2S(v)(\omega)-n_kd_2\right)\right]\\[0.25cm]
				=\,\, & \alpha\left[\limsup_{k\rightarrow\infty} W_{n_k}(c_1,d_1)(\omega)\right]+(1-\alpha)\left[\limsup_{k\rightarrow\infty} W_{n_k}(c_2,d_2)(\omega)\right]
				<\infty.
			\end{align*}
			Therefore $\mathbb{S}\left(\{n_k\},\omega\right)$ is convex. As $\Q^2$ is dense in $\R^2$, the conditions in Lemma~\ref{convx} hold for the convex function $\nu$, the convex set $\mathbb{S}\left(\{n_k\},\omega\right)$, and the point $\theta_0$. Thus for any $a\geq\theta_0$ and any $b<a\nu(\theta_0)/\theta_0$, we have $(a,b)\notin\mathbb{S}\left(\{n_k\},\omega\right)$, which implies
			\[
			\limsup_{k\rightarrow\infty}W_{n_k}(a,b)(\omega)=\infty.
			\]
			This holds for all subsequence $\{n_k\}$ and all $\omega\notin\mathcal{N}$. Hence for all $a\geq\theta_0$ and all $b<a\nu(\theta_0)/\theta_0$, we have
			\[ W_n(a,b)\rightarrow \infty \text{ a.s.} \]
		\end{proof}
	
	We recall that $W_n(\theta) =  W_n(\theta,0)$. The following corollary is a simple consequence of Proposition~\ref{wnprop}.
	\begin{cor}
		\label{wnlim}
		Almost surely
		\[
		\frac{\log W_n(\theta)}{n\theta }
		\rightarrow
		\begin{cases} 
			\frac{\nu(\theta)}{\theta} & \text{if } \theta<\theta_0\leq \infty;\\[.25cm]
			\frac{\nu(\theta_0)}{\theta_0} & \text{if }\theta_0\leq \theta<\infty.
                \end{cases}
		\]
	\end{cor}
	\begin{remark}
	\label{Rem:heuristic}
To understand why the limit in Corollary~\ref{wnlim} becomes constant for $\theta\geq\theta_0$, let us consider 
\[
\mathfrak F(\theta)=  \lim_{n\rightarrow\infty} \frac{\log W_n(\theta)}{n\theta }.
\]
 Notice that $\left[ W_n(\theta) \right]^{1/\theta}$ is indeed the $\ell_{\theta}$-norm of the sequence $\{e^{S_v}\}_{|v|=n}$. Thus, it is non-increasing in $\theta$. Therefore, $\mathfrak F(\theta)$ is also non-increasing in $\theta$. Now by the Cauchy–Schwarz inequality, we get that for any $\theta_1,\theta_2>0$,
\[
\left(W_n(\theta_1+\theta_2)\right)^2\leq W_n(2\theta_1)\cdot W_n(2\theta_2).
\]
Since dyadic rational numbers are dense in the real numbers, this gives us that for any $\alpha\in(0,1)$,
\[
W_n(\alpha\theta_1+(1-\alpha)\theta_2)\leq W_n(\theta_1)^{\alpha}\cdot W_n(\theta_2)^{1-\alpha},
\]
 which means that $\log W_n(\theta)$ is convex in $\theta$, and therefore so is $\theta \mathfrak F(\theta)$. Now, for $\theta<\theta_0$, $\mathfrak F(\theta)=\nu(\theta)/\theta$. So by Remark~\ref{Rem:nutheta/theta}, the left-derivative of $\mathfrak F$ is $0$ at $\theta_0$. Hence the right-derivative is greater than or equal to $0$ at $\theta_0$, by convexity of the function $\theta\mapsto\theta \mathfrak F(\theta)$. Using again this convexity, it is now easy to show that $\mathfrak F'(\theta)\geq 0$ for all $\theta\geq \theta_0$, hence $\mathfrak F(\theta)\geq \mathfrak F(\theta_0)$ for all $\theta\geq \theta_0$. But since $\mathfrak F$ is non-increasing, it has to be constant for $\theta\geq \theta_0$.
\end{remark}

\begin{prop}
\label{Prop:Ratio_ynwn}
	For $\theta<\theta_0\leq\infty$ or $\theta=\theta_0<\infty$,
	\begin{equation*}
		\label{ynwn}
		\frac{Y_n^{\mu}(\theta)}{W_n(\theta)} \prob \langle\mu\rangle,
	\end{equation*}
	where $\langle\mu\rangle$ is the mean of $\mu$ and $Y_n^{\mu}(\theta)$ is as defined in Corollary~\ref{Thm:Coupling}.
	\end{prop}

\begin{proof}
Recall that as in~\eqref{wnliu},  for $\theta<\theta_0\leq\infty$, 
		\begin{equation}
			\label{wnliu2}
			W_n(\theta,\nu(\theta))\rightarrow D_{\theta}^{\infty} \text{ a.s.}
		\end{equation}
	For $\theta_0<\infty$, A\"{i}d\'{e}kon and Shi~\cite{AiSh14} have shown that  under the assumptions in Section~\ref{SubSec:Assumptions},
	\begin{equation}
		\label{aidekonshi}
		\sqrt{n}\, W_n(\theta_0,\nu(\theta_0))\prob \left(\frac{2}{\pi\sigma^2}\right)^{1/2} D_{\theta_0}^{\infty},
	\end{equation}
	where $\sigma^2$ and $D_{\theta_0}^{\infty}$ are as mentioned in Section~\ref{SubSubSec:Boundary-Case-Results}. 
Also, 
		Hu and Shi ~\cite{HuSh09} have proved that under the assumptions in Section~\ref{SubSec:Assumptions}, for $\theta_0<\theta<\infty$,
		\begin{equation}
			\label{hushi}
			\frac{1}{\log n}\left(  \log W_n(\theta)- \frac{\nu(\theta_0)}{\theta_0}\theta n \right) \prob -\frac{3\theta}{2\theta_0}.
		\end{equation}	
Now, observe that 
	\[
	\frac{Y_n^{\mu}(\theta)}{W_n(\theta)} - \langle\mu\rangle
	= \sum_{|v|=n}\left(  \frac{e^{\theta S(v)}}{\sum_{|u|=n}e^{\theta S(u)}}   \right)
	\left(  Y_v- \langle\mu\rangle \right).
	\]
	We define
	\[
	M_n(\theta):= \max_{|v|=n}  \frac{e^{\theta S(v)}}{\sum_{|u|=n}e^{\theta S(u)}}
	=\frac{e^{\theta R_n}}{W_n(\theta)}.
	\]
	We recall that $W_n(a) =  W_n(a,0)$. For $\theta\in(0,\theta_0)$, we choose any $\theta_1\in(\theta,\theta_0)$. Then  we get
	\[
	M_n(\theta)
	\leq \frac{\left[ W_n(\theta_1) \right]^{\theta/\theta_1}}{W_n(\theta)}
	\leq \frac{\left[W_n(\theta_1,\nu(\theta_1))\right]^{\theta/\theta_1}\cdot e^{-n\theta\left(\frac{\nu(\theta)}{\theta}-\frac{\nu(\theta_1)}{\theta_1}\right)}}{W_n(\theta,\nu(\theta))}
	\]
	Since $\nu$ is strictly convex, $\nu(\theta)/\theta$ is strictly decreasing for $\theta\in(0,\theta_0)$. Therefore using~\eqref{wnliu2}, we get
	\begin{equation}
		\label{mnlim1}
		M_n(\theta)\rightarrow 0 \text{ a.s.}
	\end{equation}
	For  $\theta=\theta_0<\infty$, we choose any $\theta_2\in(\theta_0,\infty)$. Observe that
	\begin{equation}
		M_n(\theta_0)
		\leq \frac{\left[W_n\left( \theta_2\right) \right] ^{\theta_0/\theta_2}} {W_n(\theta_0 )}
		= \frac{\left[n^{\theta_2/\theta_0}W_n\left( \theta_2, {\theta_2\nu(\theta_0)}/{\theta_0}\right) \right] ^{\theta_0/\theta_2}} {nW_n(\theta_0,       \nu (\theta_0) )}. 
		\label{Equ:mnlim2upperbound}
	\end{equation}
	Now, using~\eqref{aidekonshi}, the denominator on the right-hand side of~\eqref{Equ:mnlim2upperbound} goes to $\infty$ in probability, and by~\eqref{hushi}, the numerator goes to $0$ in probability.
Therefore, we obtain 
		\begin{equation}
		\label{mnlim2}
		M_n(\theta_0)\prob 0.
	\end{equation}
    Let $\mathcal{F}$ be the $\sigma$-field generated by the branching random walk, and $Y\sim\mu$. 
    Then, using~\cite[Lemma~2.1]{BiKy97}, which is a particular case of~\cite[Lemma~2.2]{Kurt72}, we get that for every $0<\varepsilon<1/2$,
		\begin{align*}
			&\P\left( \left.\left| \frac{Y_n^{\mu}(\theta)}{W_n(\theta)} - \langle\mu\rangle \right|>\varepsilon\right| \mathcal{F} \right)\\
			&\leq \frac{2}{\varepsilon^2}
			\left(
			\int_{0}^{\frac{1}{M_n(\theta)}} M_n(\theta)t\cdot\P\left(  \left|Y- \langle\mu\rangle\right|>t  \right)\,\dd t+ 
			\int_{\frac{1}{M_n(\theta)}}^{\infty} \P\left(  \left|Y- \langle\mu\rangle\right|>t \right)\,\dd t
			\right),
		\end{align*}
		which by~\eqref{mnlim1},~\eqref{mnlim2} and dominated convergence theorem, converges to $0$ in probability as $n\rightarrow\infty$. Then by taking expectation and using dominated convergence theorem again, we get 
		\[
		\lim_{n\rightarrow\infty}
		\P\left( \left| \frac{Y_n^{\mu}(\theta)}{W_n(\theta)} - \langle\mu\rangle \right|>\varepsilon \right)=0.
		\]
		 This completes the proof.

\end{proof}

\begin{remark}
\label{Rem:remark_on_Mn}
We note here that Proposition ~\ref{Prop:Ratio_ynwn} holds only when $\theta<\theta_0\leq\infty$ 
or $\theta=\theta_0<\infty$. It is not clear that the conclusion of this proposition holds 
for  the ABC, that is,
when $\theta_0 < \theta < \infty$. In fact, in that case,
$e^{\theta R_n} = \Theta_{\bP}\left( W_n(\theta) \right)$ (follows from~\cite[Theorem~1.1]{Aide13} 
and the proof of Theorem~\ref{Thm:Above-Boundary-Case} given below). Thus, 
~\eqref{mnlim2} does not hold for the ABC. 
\end{remark}

	\section{Proofs of The Main Results}
	\label{Sec:Proofs}
	In this section we prove the main theorems. 
	We start by proving the 
	Centered asymptotic limits: proving first 
	Theorems~\ref{Thm:Boundary-Case} and~\ref{Thm:Below-Boundary-Case} and then
	Theorems~\ref{Thm:Boundary-Case-Version-II} 
	and~\ref{Thm:Below-Boundary-Case-Version-II}.	Proof of Theorem~\ref{Thm:Above-Boundary-Case} is given there after. 
	We then prove the almost sure asymptotic limit, 
	Theorem~\ref{Thm:SLLN}. Finally we end by proving the 
	Brunet-Derrida type results, Theorem~\ref{Thm:Point-Process-Conv}
	and Theorem~\ref{Thm:Point-Process-Weak-Conv}.

	\subsection{Proof of  Theorems ~\ref{Thm:Boundary-Case} and~\ref{Thm:Below-Boundary-Case}} 
	\begin{proof}
Proposition~\ref{Prop:Ratio_ynwn}, together with~\eqref{wnliu2}, gives us that for $\theta<\theta_0\leq\infty$, 
		\begin{equation}
			Y_n^{\mu}(\theta)\cdot e^{-n\nu(\theta)}\prob  D_{\theta}^{\infty} \cdot\langle \mu\rangle.
		\end{equation}
This implies
\begin{equation}
\label{Equ:yn_centered1}
			\log Y_n^{\mu}(\theta)-\log E -n\nu(\theta) \prob  \log D_{\theta}^{\infty} -\log E +\log \langle \mu\rangle,
		\end{equation}
		where $E \sim \Exponential{1}$ and is  independent of $\{Y_v : |v|=n\}_{n\geq 0}$ and also independent of the BRW. 
Similarly, combining Proposition~\ref{Prop:Ratio_ynwn} and~\eqref{aidekonshi}, we obtain that 
\begin{equation}
		Y_n^{\mu}(\theta_0)\cdot\sqrt{n}\cdot e^{-n\nu(\theta_0)}\prob \left(\frac{2}{\pi\sigma^2}\right)^{1/2}\cdot D_{\theta_0}^{\infty}\cdot\langle \mu\rangle,
	\end{equation}
which implies
\begin{equation}
\label{Equ:yn_centered2}
			\log Y_n^{\mu}(\theta_0)-\log E -n\nu(\theta_0)+ \frac{1}{2}\log n \prob  \frac{1}{2}\log  \left(\frac{2}{\pi\sigma^2}\right)+ \log D_{\theta_0}^{\infty} -\log E +\log \langle \mu\rangle.
\end{equation}
Now, combining~\eqref{Equ:yn_centered1} and~\eqref{Equ:yn_centered2} together with Corollary~\ref{Thm:Coupling} gives us the required result.	
	\end{proof}
	
		\subsection{Proof of  Theorems  ~\ref{Thm:Boundary-Case-Version-II} and~\ref{Thm:Below-Boundary-Case-Version-II}} 
		\begin{proof} By using a similar argument as in~\eqref{Equ:Tranforming_Relation_Proof}, we observe that
		 \begin{align*}
				\theta R_n^*(\theta,\mu)-\log Y_n^{\mu}(\theta)
				&=\max_{|v|=n}\left(\theta S(v)+\log Y_v-\log E_v\right)-\log \left(\sum_{|u|=n} e^{\theta S(u)}Y_u\right)\\[.25cm]
				&=-\log\left(\min_{|v|=n} E_v\left( \frac{e^{\theta S(v)}Y_v}{ \sum_{|u|=n} e^{\theta S(u)}Y_u } \right)^{-1}\right)\\[.25cm]
				&\ed -\log E,
			\end{align*}
where $E\sim\Exponential{1}$. Now, using Proposition~\ref{ynwn}, we obtain that for $\theta<\theta_0\leq\infty$ or $\theta=\theta_0<\infty$,
\[
\theta R_n^*(\theta,\mu)-\log W_n(\theta) \cd \log \langle\mu\rangle -\log E.
\]
This, together with~\eqref{wnliu2} and~\eqref{aidekonshi} completes the proof.
\end{proof}
		
	\subsection{Proof of  Theorem~\ref{Thm:Above-Boundary-Case}}

 \begin{proof}
		From~\cite[Theorem~2.3]{Mada17}, it follows that  under our assumptions, 		for $\theta_0 < \theta < \infty$, there exists a positive random variable  
		$\mathfrak{D}_{\theta}$, which 
		may depend on $\theta$,  such that,
		\begin{equation}
			\label{Equ:mada_result} 
			\log W_n(\theta)- \frac{\nu(\theta_0)}{\theta_0}\theta n + \frac{3\theta}{2\theta_0} \log n
\cd 
\log\mathfrak{D}_{\theta} + \frac{\theta}{\theta_0} \log D_{\theta_0}^{\infty},
\end{equation}	
where $\mathfrak{D}_{\theta}$ is independent of $D_{\theta_0}^{\infty}$. 
		Since $W_n(\theta)=Y_n^{\delta_1}(\theta)$, using Corollary~\ref{Thm:Coupling}, we get that for $\theta_0 < \theta < \infty$,
		\begin{equation}
			R_n^*(\theta) - \frac{\nu\left(\theta_0\right)}{\theta_0} n + \frac{3}{2 \theta_0} 
			\log n 
			\cd \frac{1}{\theta_0} \log D_{\theta_0}^{\infty}+ \frac{1}{\theta}\log\mathfrak{D}_{\theta} - \frac{1}{\theta}\log E,
			\label{Equ:H-theta-infty_Evaluate_theta>theta0_1}
		\end{equation}
		where $E\sim\Exponential{1}$. We write the limiting random variable as $H_{\theta}^{\infty}$.
		Now, for $u$ such that $|u|=1$, we define
		\[
		R_{n-1}^{*(u)}(\theta):=\left(\max_{v>u, |v|=n} S(v)-\frac{1}{\theta}\log E_v\right)- S(u).
		\]
		Note that $\{R_{n-1}^{*(u)}(\theta)\}_{|u|=1}$ are i.i.d. and have the same distribution as $R_{n-1}^*(\theta)$. Now,
		\begin{align*}
			R_{n}^*(\theta)
			&=\max_{|u|=1}\left(\max_{v>u, |v|=n}S(v)-\frac{1}{\theta}\log E_v\right)\\
			&=\max_{|u|=1}\left( S(u)+R_{n-1}^{*(u)}(\theta)\right).
		\end{align*}	
		This implies
		\begin{align}
			&R_{n}^*(\theta) - \frac{\nu\left(\theta_0\right)}{\theta_0} n + \frac{3}{2 \theta_0} \log n \nonumber\\
			=\,&\max_{|u|=1}\left( S(u)-\frac{\nu\left(\theta_0\right)}{\theta_0}+R_{n-1}^{*(u)}(\theta) -\frac{\nu\left(\theta_0\right)}{\theta_0}(n-1 ) + \frac{3}{2 \theta_0} \log n  \right).
			\label{Equ:centered_recursion_theta}
		\end{align}	
		For $\theta_0 < \theta < \infty$, let $G_{\theta,n}$ be the distribution function of $R_{n}^*(\theta) - \frac{\nu\left(\theta_0\right)}{\theta_0} n + \frac{3}{2 \theta_0} \log n$, and it converges pointwise to $G_{\theta}$. Now,~\eqref{Equ:centered_recursion_theta} tells us that
		\begin{align}
    G_{\theta,n}(x) = \E\left[ \prod_{|u|=1} G_{\theta,n-1}\left( x -S(u)+\frac{\nu(\theta_0)}{\theta_0}  + \frac{3}{2 \theta_0} \log\left( 1-\frac{1}{n} \right) \right)
    \right],
\end{align}
which implies
	\begin{align}
    G_{\theta}(x) = \E\left[ \prod_{|u|=1} G_{\theta}\left( x -S(u)+\frac{\nu(\theta_0)}{\theta_0}   \right)
    \label{Equ:G-theta-equation}
    \right]
\end{align}
	If we define $g_{\theta}:(0,\infty)\rightarrow [0,1]$ as $g_{\theta}(t)= G_{\theta}(-\log t)$, then from~\eqref{Equ:G-theta-equation} we have 
\begin{align}
   g_{\theta}(t) &= \E\left[ \prod_{|u|=1} g_{\theta}\left(t e^{S(u)-\frac{\nu(\theta_0)}{\theta_0}}\right) 
    \right].
     \label{Equ:g-theta-equation}
\end{align}
Now, if $G_{\theta_0,n}$ is the distribution function of $R_{n}^*(\theta_0) - \frac{\nu\left(\theta_0\right)}{\theta_0} n + \frac{1}{2 \theta_0} \log n$, and it converges pointwise to $G_{\theta_0}$, then by defining $g_{\theta_0}:(0,\infty)\rightarrow [0,1]$ as $g_{\theta_0}(t)= G_{\theta_0}(-\log t)$, a similar argument  gives us  
\begin{align}
   g_{\theta_0}(t) &= \E\left[ \prod_{|u|=1} g_{\theta_0}\left(t e^{S(u)-\frac{\nu(\theta_0)}{\theta_0}}\right) 
    \right].
     \label{Equ:g-theta0-equation}
\end{align}
Since both $g_{\theta}$ and $g_{\theta_0}$ are non-degenerate survival functions,~\eqref{Equ:g-theta-equation} and \eqref{Equ:g-theta0-equation},  in conjunction with~\cite[Theorem~1.1]{AlBiMe12}, imply that $g_{\theta}(t) = g_{\theta_0}(te^{c_{\theta}})$, for some $c_{\theta}\in\R$. Consequently, we get $G_{\theta}(x) = G_{\theta_0}(x-c_{\theta})$, which means 
\begin{align}
   H_{\theta}^{\infty}\disteq H_{\theta_0}^{\infty}+ c_{\theta}.
   \label{Equ:H-theta-infty_Evaluate_theta>theta0_2}
\end{align}
This completes the proof.	
	\end{proof}

\begin{proof}[An alternative proof]
 From~\cite[Theorem~1]{BaRhVa12}, we know
\begin{align}
\E[e^{-t\mathfrak{D}_{\theta}}] = \begin{cases} 
			e^{-\left(a_\theta t\right)^{\theta_0/\theta}} & \text{if } t\geq0;\\
			\infty & \text{if } t<0,
			\end{cases}		
			\label{Equ:mathfrak-D_Laplace}
\end{align}
for some $a_{\theta}>0$.
An alternative  way to derive~\eqref{Equ:H-theta-infty_Evaluate_theta>theta0_2} from~\eqref{Equ:H-theta-infty_Evaluate_theta>theta0_1} is to show that 
\begin{align}
\frac{\mathfrak{D}_{\theta}}{E}\disteq \frac{a_{\theta}}{E^{\theta/\theta_0}}.
\label{Equ:H-theta-infty_Evaluate_theta>theta0_3}
\end{align}
Using an argument similar to that in Example
9.17 of~\cite{StHa04}, together with~\eqref{Equ:mathfrak-D_Laplace}, we obtain that for any $x>0$,
\begin{align*}
\P\left(\frac{a_\theta E}{\mathfrak{D}_{\theta}}>x\right)=\E\left[\P\left(\left. E> \frac{x \mathfrak{D}_{\theta}}{a_\theta} \right|\mathfrak{D}_{\theta}\right)\right]
=\E\left[e^{-\frac{x \mathfrak{D}_{\theta}}{a_\theta}}\right]
=e^{-x^{\theta_0/\theta}}
=\P\left( E^{\theta/\theta_0}> x\right).
\end{align*}
This proves~\eqref{Equ:H-theta-infty_Evaluate_theta>theta0_3}, which implies~\eqref{Equ:H-theta-infty_Evaluate_theta>theta0_2}.
\end{proof}

	\subsection{Proof of  Theorem~\ref{Thm:SLLN}}
	\begin{proof}
		\textit{(Upper bound)}.
		Take any $\theta>0$ and let $\beta=\min(\theta,\theta_0)$. Using Markov's inequality, we get that for every $\epsilon>0$,
		\[
		\P\left( \frac{R_n^*(\theta,\mu)}{n}-\frac{\nu(\beta)}{\beta}>\epsilon \right)\leq
		e^{-n\left(\beta\epsilon+\nu(\beta)\right)/2}\cdot\E\left[e^{\beta R_n^*(\theta,\mu)/2}\right].
		\]
		Now, using Corollary~\ref{Thm:Coupling},  we have
		\begin{align*}
			\E\left[e^{\beta R_n^*(\theta,\mu)/2}\right]
			&=\E\left[ \left(\sum_{|v|=n}e^{\theta S(v)} Y_v\right)^{\beta/(2\theta)} \right]\cdot\E\left[E^{-\beta/(2\theta)}\right]\\[.25cm]
			&\leq \E\left[ \sqrt{\sum_{|v|=n}e^{\beta S(v)} Y_v^{\beta/\theta}} \right]\cdot\Gamma\left(1-\frac{\beta}{2\theta}\right)\\[.25cm]
			&\leq \sqrt{\E\left[ \sum_{|v|=n}e^{\beta S(v)} Y_v^{\beta/\theta} \right]}\cdot\Gamma\left(1-\frac{\beta}{2\theta}\right)\\[.25cm]
			&=\sqrt{e^{n\nu(\beta)}\cdot\langle\mu\rangle_{\beta/\theta}}\cdot\Gamma\left(1-\frac{\beta}{2\theta}\right),
		\end{align*}
		where $\langle\mu\rangle_{\beta/\theta}$ is the $(\beta/\theta)$-th moment of $\mu$. So for every $\epsilon>0$, we have
		\begin{equation}
			\label{bc+}
			\sum_{n=1}^{\infty} \P\left( \frac{R_n^*(\theta,\mu)}{n}-\frac{\nu(\beta)}{\beta}>\epsilon \right)<\infty.
		\end{equation}
		Therefore using the Borel-Cantelli Lemma, we obtain for all $\theta>0$, almost surely
		\begin{equation}
			\label{limsuprnn}
			\limsup_{n\rightarrow\infty} \frac{R_n^*(\theta,\mu)}{n}\leq 
			\begin{cases}
				\frac{\nu(\theta)}{\theta}&\text{if }\theta<\theta_0\leq\infty;\\[.25cm]
				\frac{\nu(\theta_0)}{\theta_0}&\text{if }\theta_0\leq\theta<\infty.
                         \end{cases}
		\end{equation}

		\noindent\textit{(Lower bound).}
		For $u$ such that $|u|=m\leq n$, we define
		\[
		R_{n-m}^{*(u)}(\theta,\mu):=\left(\max_{v>u, |v|=n} S(v)+\frac{1}{\theta}\log (Y_v/E_v)\right)- S(u).
		\]
		Note that $\{R_{n-m}^{*(u)}(\theta,\mu)\}_{|u|=m}$ are i.i.d. and have the same distribution as $R_{n-m}^*(\theta,\mu)$. Now,
		\begin{align*}
			R_{n}^*(\theta,\mu)
			&=\max_{|u|=m}\left(\max_{v>u, |v|=n}S(v)+\frac{1}{\theta}\log (Y_v/E_v)\right)\\
			&=\max_{|u|=m}\left( S(u)+R_{n-m}^{*(u)}(\theta,\mu)\right)\\
			&\geq  S(\tilde{u}_m)+\max_{|u|=m}\left(R_{n-m}^{*(u)}(\theta,\mu)\right), 
		\end{align*}	
		where
		\[
		\tilde{u}_m:=\arg\max_{|u|=m}\left(R_{n-m}^{*(u)}(\theta,\mu)\right).
		\]	
		Now, for any $\epsilon\in(0,1)$ and for $\theta<\theta_0\leq\infty$ or $\theta=\theta_0<\infty$,
		\begin{align*}
			&\P\left( \frac{R_n^*(\theta,\mu)}{n}-\frac{\nu(\theta)}{\theta}<-\epsilon \right)\\
			\leq\,\, & \P\left(S(\tilde{u}_{[\sqrt{n}]})+\max_{|u|=[\sqrt{n}]}\left(R_{n-[\sqrt{n}]}^{*(u)}(\theta,\mu)\right)<n\left(\frac{\nu(\theta)}{\theta}-\epsilon\right)\right) \\
			\leq \,\,&
			\P\left(\max_{|u|=[\sqrt{n}]}\left(R_{n-[\sqrt{n}]}^{*(u)}(\theta,\mu)\right)<n\left(\frac{\nu(\theta)}{\theta}-\frac{\epsilon}{2}\right)\right)
			+
			\P\left(S(\tilde{u}_{[\sqrt{n}]})<-\frac{n\epsilon}{2}\right)\\
			\leq\,\, &
			\E\left[\P\left(R_{n-[\sqrt{n}]}^*(\theta,\mu)<n\left(\frac{\nu(\theta)}{\theta}-\frac{\epsilon}{2}\right)\right)^{N_{[\sqrt{n}]}}\right]+
			e^{-n\epsilon\vartheta/4}\cdot\E\left[e^{-\vartheta S(\tilde{u}_{[\sqrt{n}]})/2}\right].
		\end{align*}
		Now, Corollary~\ref{wnlim}, together with Corollary~\ref{Thm:Coupling} and Proposition~\eqref{ynwn}, implies that  for $\theta<\theta_0$ and also for $\theta=\theta_0<\infty$,
		\[
		\frac{R_{n}^*(\theta,\mu)}{n}\xlongrightarrow{p} \frac{\nu(\theta)}{\theta}.
		\]
		Therefore for all large enough $n$,
		\[
		\P\left(R_{n-[\sqrt{n}]}^*(\theta,\mu)<n\left(\frac{\nu(\theta)}{\theta}-\frac{\epsilon}{2}\right)\right)<\epsilon.
		\]
		Observe, $N_{[\sqrt{n}]} < n$ implies at least $[\sqrt{n}]-\lceil\log_2n\rceil$ many particles have given birth to only one offspring. Therefore
		\[
		\P\left(N_{[\sqrt{n}]} < n \right)\leq \left(\P(N=1)\right)^{[\sqrt{n}]-\lceil\log_2n\rceil}.
		\]
		For the second term, we have
		\[
		\E\left[e^{-\vartheta S(\tilde{u}_{[\sqrt{n}]})/2}\right]
		\leq \E\left[ W_{[\sqrt{n}]}(-\vartheta/2)\right]
		=e^{[\sqrt{n}]\nu(-\vartheta/2)}.
		\]
		Therefore we have for all large enough $n$,
		\begin{align*}
			&\P\left( \frac{R_n^*(\theta,\mu)}{n}-\frac{\nu(\theta)}{\theta}<-\epsilon \right)\\
			\leq \,& \epsilon^n+\left(\P(N=1)\right)^{[\sqrt{n}]-\lceil\log_2n\rceil}
			+e^{-n\epsilon\vartheta/4+[\sqrt{n}]\nu(-\vartheta/2)}.
		\end{align*}
		Since for every $\epsilon\in(0,1)$,
		\begin{equation}
			\label{bc-}
			\sum_{n=1}^{\infty} \P\left( \frac{R_n^*(\theta,\mu)}{n}-\frac{\nu(\theta)}{\theta}<-\epsilon \right)<\infty,
		\end{equation}
		using the  Borel-Cantelli Lemma, we obtain that for $0<\theta<\theta_0$ or $\theta=\theta_0<\infty$,
		\begin{equation}
			\label{liminfrnn1}
			\liminf_{n\rightarrow\infty} \frac{R_n^*(\theta,\mu)}{n}\geq \frac{\nu(\theta)}{\theta} \text{ a.s.}
		\end{equation}
		
	To get an appropriate  lower bound for $\theta_0<\theta<\infty$, we need the 
	following result, the proof of this is given at the end of this proof. 
	
	\begin{prop}
		\label{ynlim}
		For any positively supported probability $\mu$ with finite mean, almost surely
		\[
		\frac{\log Y_n^{\mu}(\theta)}{n\theta }
		\rightarrow
		\begin{cases} 
			\frac{\nu(\theta)}{\theta} & \text{if } \theta<\theta_0\leq\infty;\\[.25cm]
			\frac{\nu(\theta_0)}{\theta_0} & \text{if }\theta_0\leq \theta<\infty.
                  \end{cases}
		\]
	\end{prop}

	Now observe that,
		\[ \theta R_n^*(\theta,\mu)=\max_{|v|=n}\left( \theta S(v)+\log Y_v-\log E_{v} \right)\geq \max_{|v|=n}\left( \theta S(v)+\log Y_v\right)-\log E_{v_n}, \]
		where 
		\[
		v_n=\arg\max_{|v|=n}\left( \theta S(v)+\log Y_v\right).
		\]
		Observe
		\[
		Y_n^{\mu}(\theta+\theta_0)=\sum_{|v|=n}e^{(\theta+\theta_0)S(v)}Y_v\leq
		W_n(\theta_0)\cdot e^{\max_{|v|=n}\left( \theta S(v)+\log Y_v\right)}.
		\]
		Therefore we have
		\begin{align*}
			\frac{\theta R_n^*(\theta,\mu)}{n}\geq \frac{\log Y_n^{\mu}(\theta+\theta_0)}{n}-\frac{\log W_n(\theta_0)}{n}-\frac{\log E_{v_n}}{n}.
		\end{align*}
		Since  $\E[|\log E_{v_n}|]$ is finite,  the Borel-Cantelli Lemma implies that the last terms on the right hand side converges to $0$ a.s. By Corollary~\ref{wnlim} and Proposition~\ref{ynlim}, the first and the second terms a.s. converges to $ (\theta+\theta_0)\nu(\theta_0)/\theta_0$ and $\nu(\theta_0)$ respectively. Thus whenever $\theta_0<\infty$,  we obtain that for all $\theta>\theta_0$,
		\begin{equation}
			\label{liminfrnn2}
			\liminf_{n\rightarrow\infty} \frac{R_n^*(\theta,\mu)}{n}\geq \frac{\nu(\theta_0)}{\theta_0} \text{ a.s.}
		\end{equation}
		This, together with~\eqref{limsuprnn} and~\eqref{liminfrnn1}, completes the proof.
	\end{proof}
	
	\subsubsection{Proof of  Proposition~\ref{ynlim}}
	\begin{proof}
		Corollary~\ref{Thm:Coupling} says that
		\[\theta R_n^*(\theta,\mu)\xlongequal{ d } \log Y_n^{\mu}(\theta)-\log E .\]
		Since $\E[|\log E|]<\infty$,~\eqref{bc+} and~\eqref{bc-},  together with the  Borel-Cantelli Lemma, imply that for $\theta<\theta_0\leq \infty$ and also for $\theta=\theta_0<\infty$,
		\[
		\frac{\log Y_n^{\mu}(\theta)}{n\theta }
		\rightarrow
		\frac{\nu(\theta)}{\theta} \text{ a.s.}
		\]
		and for $\theta_0<\theta<\infty$, 
		\[
		\limsup_{n\rightarrow\infty} \frac{\log Y_n^{\mu}(\theta)}{n\theta }\leq \frac{\nu(\theta_0)}{\theta_0} \text{ a.s.}
		\]
		So	for any $a>0$ and $b\in\R$, we have almost surely
		\[
		Y_n^{\mu}(a,b): = Y_n^{\mu}(a)\cdot e^{-nb}
		\rightarrow
		\begin{cases} 
			0 & \text{if } a<\theta_0,\, b>\nu(a);\\
			\infty & \text{if } a<\theta_0,\, b<\nu(a);\\
			0 & \text{if } \theta_0<\infty,\,a\geq\theta_0,\, b> a\nu(\theta_0)/\theta_0.
                 \end{cases}
		\]
		Now, the exact similar argument as in the proof of Proposition~\ref{wnprop}{(v)} implies that
		for $\theta_0<\infty$, $a\geq\theta_0$ and $b< a\nu(\theta_0)/\theta_0$,
		\[
		Y_n^{\mu}(a,b)
		\rightarrow \infty \text{ a.s.}
		\]
		Hence for $\theta_0<\theta<\infty$, 
		\[
		\frac{\log Y_n^{\mu}(\theta)}{n\theta }
		\rightarrow
		\frac{\nu(\theta_0)}{\theta_0} \text{ a.s.}
		\]
		This proves the proposition.
	\end{proof}

\subsection{Proof of  Theorem~\ref{Thm:Point-Process-Conv}}	
\label{SubSec:Proof-of-BD-results}
\begin{proof}
R\'{e}nyi’s representation~\cite{Reny53}, together with the generalized version of it by Tikhov (see equation~(3) of Tikhov~\cite{Tikh92}), gives us the following lemma.  
\begin{lemma}
	\label{ExpArray}
	Let $\left\{E_{i,n}:1\leq i\leq m_n,n\geq 1\right\}$ be an array of independent random variables with $E_{i,n}\sim \Exponential{ \lambda_{i,n} }$. Suppose  for all $n\geq1$,
	$\sum_{i=1}^{m_n}\lambda_{i,n}=1$, and $\lim_{n\rightarrow\infty}\max_{i=1}^{m_n}\lambda_{i,n}=0$.
	 Then as $n\rightarrow\infty$, the point process
	 \[  \sum_{i=1}^{m_n} \delta_{E_{i,n}}\xrightarrow{d} \Ni,	 \]
	 where $\Ni$ is a homogeneous Poisson point process on $\R_+$ with intensity $1$.
\end{lemma}

Now, let $\Fi$ be the $\sigma$-algebra generated by the branching random walk.  We know that, conditionally on $\Fi$,  $\left\{ E_vW_n(\theta)e^{-\theta S(v)} \right\}$ are independent. Furthermore, conditionally on $\Fi$, $E_vW_n(\theta)e^{-\theta S(v)}$ follows $\Exponential{ \frac{e^{\theta S(v)}}{W_n(\theta)}  }$. Note that
\[ \sum_{|v|=n} \frac{e^{\theta S(v)}}{W_n(\theta)} =1 ,\]
and by~\eqref{mnlim1} and~\eqref{mnlim2}, we also have that  for $\theta<\theta_0\leq\infty$ or $\theta=\theta_0<\infty$,
\[ \max_{|v|=n} \frac{e^{\theta S(v)}}{W_n(\theta)}\prob 0. \]
Therefore by Lemma~\ref{ExpArray}, for any positive integer $k$, Borel sets $B_1,B_2,\ldots,B_k$ and  non-negative integers $t_1,t_2,\ldots,t_k$, we have
\begin{align*}
&\P\left(\left. \sum_{|v|=n} \delta_{E_vW_n(\theta)e^{-\theta S(v)}}(B_1)=t_1,\ldots, \sum_{|v|=n} \delta_{E_vW_n(\theta)e^{-\theta S(v)}}(B_k)=t_k \right| \Fi\right)\\[.25cm]
&\qquad\qquad\qquad\qquad\qquad\qquad\qquad\qquad\qquad\qquad \prob \P\left(\Ni(B_1)=t_1,\ldots,\Ni(B_k)=t_k\right).
\end{align*}
Then, using the dominated convergence theorem, we get
\begin{align*}
	&\P\left( \sum_{|v|=n} \delta_{E_vW_n(\theta)e^{-\theta S(v)}}(B_1)=t_1,\ldots, \sum_{|v|=n} \delta_{E_vW_n(\theta)e^{-\theta S(v)}}(B_k)=t_k \right)\\[.25cm]
	&\qquad\qquad\qquad\qquad\qquad\qquad\qquad\qquad\qquad\qquad \rightarrow \P\left(\Ni(B_1)=t_1,\ldots,\Ni(B_k)=t_k\right).
\end{align*}
or equivalently (see Theorem~11.1.VII of Daley and Vere-Jones~\cite{DaVe08}),
\begin{equation}
\sum_{|v|=n} \delta_{E_vW_n(\theta)e^{-\theta S(v)}}\xrightarrow{d} \Ni. \label{Eqn:Daley-Vere-Jones}
\end{equation}

Now for $\Ni=\sum_{j\geq1}\delta_{\mathfrak z_j}$,  we take $\Yi=\sum_{j\geq1}\delta_{-\log \mathfrak z_j}$. Clearly, $\Yi$ is an inhomogeneous Poisson point process on $\R$ with intensity measure $e^{-x}\,\dd x$. 
Since $-\log(.)$ is continuous and therefore Borel measurable,~\eqref{Eqn:Daley-Vere-Jones} implies that
\begin{equation}
\Ui_n:=\sum_{|v|=n} \delta_{\theta S_v-\log E_v-\log W_n(\theta)}\xrightarrow{d} \Yi. \label{uny}
\end{equation}
To simplify the notations, for all $\theta<\theta_0\leq\infty$, we denote
\[A_n(\theta)= n\nu(\theta)+\log D_{\theta}^{\infty},\]
and for $\theta=\theta_0<\infty$, we denote
\[A_n(\theta_0)=n\nu(\theta_0)-\frac{1}{2}\log n+\log D_{\theta_0}^{\infty}+\frac{1}{2}\log\left(\frac{2}{\pi\sigma^2}\right).\]
Recall that by~\eqref{wnliu2} and~\eqref{aidekonshi}, for $\theta<\theta_0\leq\infty$ or $\theta=\theta_0<\infty$,
\[ A_n(\theta)-\log W_n(\theta)\prob0. \]
Now, take any positive integer $k$, non-negative integers $\{t_i\}_{i=1}^k$, and extended real numbers $\{a_i\}_{i=1}^k$ and $\{b_i\}_{i=1}^k$ with $a_i<b_i$ for all $i$. We choose $\delta \in \left(0, \min_{i=1}^k (b_i-a_i)/2\right)$. Then, we have
\begin{align*}
	&\P\left( \Ui_n\left((a_1-\delta,b_1+\delta)\right)\leq t_1,\ldots,\Ui_n\left((a_k-\delta,b_k+\delta)\right)\leq t_k  \right)\\[.25cm] &\qquad-\P\left(\left|A_n(\theta)-\log W_n(\theta)\right|>\delta \right)\\[.25cm]
	\leq\,\,& \P\left( Z_n(\theta)\left((a_1,b_1)\right)\leq t_1,\ldots,Z_n(\theta)\left((a_k,b_k)\right)\leq t_k  \right)\\[.25cm]
	\leq\,\,& \P\left( \Ui_n\left((a_1+\delta,b_1-\delta)\right)\leq t_1,\ldots,\Ui_n\left((a_k+\delta,b_k-\delta)\right)\leq t_k  \right)\\[.25cm]
	&\qquad +\P\left(\left|A_n(\theta)-\log W_n(\theta)\right|>\delta \right).
\end{align*}
Now, by~\eqref{uny}, we have $\Ui_n\xrightarrow{d}\Yi$. Since $\Yi$ is a  Poisson point process,  it is continuous. Therefore, allowing $n\rightarrow\infty$ and then letting $\delta\rightarrow0$, we obtain
\begin{align*}
&\lim_{n\rightarrow\infty}\P\left( Z_n(\theta)\left((a_1,b_1)\right)\leq t_1,\ldots,Z_n(\theta)\left((a_k,b_k)\right)\leq t_k  \right)\\[.25cm]
&\qquad\qquad\qquad\qquad= \P\left( \Yi\left((a_1,b_1)\right)\leq t_1,\ldots,\Yi\left((a_k,b_k)\right)\leq t_k  \right),
\end{align*}
or equivalently,
$Z_n(\theta)\xrightarrow{d}\Yi$. 
This completes the proof.
\end{proof}

\subsection{Proof of  Theorem~\ref{Thm:Point-Process-Weak-Conv}}
\begin{proof}
	This is a slightly weaker version. It follows from arguments similar to those of the proof of the Theorem~\ref{Thm:Point-Process-Conv}.
\end{proof}

\section*{Appendix}
\setcounter{theorem}{0}
    \renewcommand{\thetheorem}{A.\arabic{theorem}}
     \begin{prop}
      \label{Prop:Equ-1.2}
     Under assumptions {\bf (A1)} and {\bf (A3)}, there exists $q > 0$ such that~\eqref{asmp_rem} holds.
     \end{prop}
	\begin{proof}
		Observe that
		\[ \int_{\R}e^{\theta x}\,Z(\dd x)\leq Ne^{\theta\left(\max_{j=1}^N\xi_{j}\right)}, \quad\text{ and }\quad e^{\theta\left(\max_{j=1}^N\xi_{j}\right)}\leq  \int_{\R}e^{\theta x}\,Z(\dd x).\]
		Now, using H\"{o}lder's inequality, we have 
		\begin{align*}
			\,\,\E\left[ \left( \int_{\R}e^{\theta x}\,Z(\dd x)\right)^{1+q} \right]
			\leq & \,\, \E\left[  N^{1+q}\cdot e^{\theta (1+q) \left(\max_{j=1}^N\xi_{j}\right)} \right]\\
			\leq & \left(\E\left[N^{(1+q)^2}\right]\right)^{\frac{1}{1+q}}\cdot
			\left(\E\left[e^{\theta \left(\frac{(1+q)^2}{q}\right) \left(\max_{j=1}^N\xi_{j}\right)}\right]\right)^{\frac{q}{1+q}}\\
			\leq & \left(\E\left[N^{(1+q)^2}\right]\right)^{\frac{1}{1+q}}\cdot
			\left(\E\left[\int_{\R}e^{\theta\left(\frac{(1+q)^2}{q}\right) x}\,Z(\dd x)\right]\right)^{\frac{q}{1+q}}\\
			= & \left(\E\left[N^{(1+q)^2}\right]\right)^{\frac{1}{1+q}}\cdot
			\left(m\left(\theta(1+q)^2/q\right)\right)^{\frac{q}{1+q}}.
		\end{align*}
		Then, by choosing $q$ such that $(1+q)^2\leq1+p$, one gets~\eqref{asmp_rem}.
		\end{proof}

	\begin{prop}
	\label{Prop:Convexity_of_nu}
		Under assumptions {\bf (A1)} and {\bf (A2)}, the function 
		$\theta \mapsto \nu(\theta)$ 
		is strictly convex inside the open interval $(-\vartheta,\infty)$.
	\end{prop}
	\begin{proof}
		From Assumption {\bf (A1)}, we know that 
		\[
		m(\theta):=\E\left[\int_{\R} e^{\theta x}\,Z(\dd x) \right]<\infty,
		\]
		for all $\theta\in(-\vartheta,\infty)$. Therefore using dominated convergence theorem, we have for all $\theta\in(-\vartheta,\infty)$,
		\[
		m'(\theta)=\E\left[\int_{\R} xe^{\theta x}\,Z(\dd x) \right]<\infty,
		\]
		and
		\[
		m''(\theta)=\E\left[\int_{\R} x^2e^{\theta x}\,Z(\dd x) \right]<\infty.
		\]
		From Assumption {\bf (A2)}, we have that $\P(Z(\{t\})=N)<1$ for all $t\in\R$. Therefore for all $t\in\R$,
		\begin{align*}
			& \E\left[\int_{\R} (x-t)^2e^{\theta x}\,Z(\dd x) \right]>0\\[.25cm]
			\Rightarrow\,\, &    \E\left[\int_{\R} x^2e^{\theta x}\,Z(\dd x) \right]-2t\E\left[\int_{\R} xe^{\theta x}\,Z(\dd x) \right]+t^2\E\left[\int_{\R} e^{\theta x}\,Z(\dd x) \right]>0\\[.25cm]
			\Rightarrow \,\,& \E\left[\int_{\R} x^2e^{\theta x}\,Z(\dd x) \right]\cdot\E\left[\int_{\R} e^{\theta x}\,Z(\dd x) \right]>\left(\E\left[\int_{\R} xe^{\theta x}\,Z(\dd x) \right]\right)^2\\[.25cm]
			\Rightarrow \,\,& m''(\theta)m(\theta)>\left(m'(\theta)\right)^2.
		\end{align*}
		Hence we have for all $\theta\in(-\vartheta,\infty)$,
		\[
		\nu''(\theta)=\frac{m''(\theta)m(\theta)-\left(m'(\theta)\right)^2}{\left(m(\theta)\right)^2}>0.
		\]
		This proves the proposition.
	\end{proof}

\section*{Statement on No Conflict of Interest}
We declare that we have no competing interests as defined by Springer, the publisher, 
or other interests that might be perceived to influence the results and/or discussion reported in this paper.

\section*{Data Availability}
This is an article which is based on a theoretical research. We thus 
have no research data to declare related to this work. 

\section*{Acknowledgment}
The authors are grateful to the anonymous  
reviewer for his/her very insightful remarks, which have vastly improved the quality of the exposition.

\bibliographystyle{plain}

\end{document}